\documentclass{aims}

\usepackage{amsmath}

\usepackage{amssymb,amsthm,amsmath,pictex,color}

\font\ninerm=cmr9
\font\eightrm=cmr8
\font\sevenrm=cmr7

\def\real{\mathbb{R}}
\def\dH#1{\dot H^{#1}(\Omega)}
\def\uT{g}
\def\uTdel{g^\delta}
\def\uTtd{\tilde{g}^\delta}
\definecolor{ForestGreen}{rgb}{0.13,0.55,0.13}
\definecolor{OliveGreen}{rgb}{0.33,0.55,0}
\def\OliveGreen#1{{\color{OliveGreen}#1}}
\def\Blue#1{{\color{blue}#1}}
\def\Red#1{{\color{red}#1}}
\def\Black#1{{\color{black}#1}}

  \usepackage{paralist}
  \usepackage{graphics} 
  \usepackage{epsfig} 
\usepackage{graphicx}  \usepackage{epstopdf}
 \usepackage[colorlinks=true]{hyperref}
\hypersetup{urlcolor=blue, citecolor=red}

\def\currentvolume{X}
 \def\currentissue{X}
  \def\currentyear{200X}
   \def\currentmonth{XX}
    \def\ppages{X--XX}
     \def\DOI{10.3934/xx.xx.xx.xx}

\newtheorem{theorem}{Theorem}[section]

\newtheorem{lemma}[theorem]{Lemma}

\theoremstyle{definition}

\newtheorem{remark}{Remark}

\newcommand{\Margin}[1]{}
\newcommand{\revision}[1]{#1}

\title[Regularization by fractional operators] 
      {Regularization of a backwards parabolic equation by fractional operators}

\author[Barbara Kaltenbacher and William Rundell]{}


 \email{barbara.kaltenbacher@aau.at}
 \email{rundell@math.tamu.edu}

\thanks{The work of the first author was supported by the Austrian Science Fund FWF
under the grants I2271 and P30054 as well as partially by the
Karl Popper Kolleg ``Modeling-Simulation-Optimization'',
funded by the Alpen-Adria-Universit\" at Klagenfurt and by the
Carin\-thian Economic Promotion Fund (KWF).\\
The work of the second author was supported 
in part by the
National Science Foundation through award DMS-1620138.}

\thanks{$^*$ corresponding author}

\begin{document}
\maketitle

\centerline{\scshape Barbara Kaltenbacher$^*$}
\medskip
{\footnotesize
 \centerline{Department of Mathematics}
   \centerline{Alpen-Adria-Universit\"at Klagenfurt}
   \centerline{9020 Klagenfurt, Austria}
} 

\medskip

\centerline{\scshape William Rundell}
\medskip
{\footnotesize
 \centerline{ Department of Mathematics}
   \centerline{Texas A\&M University}
   \centerline{College Station, Texas 77843, USA}
}

\bigskip

 \centerline{(Communicated by the associate editor name)}

\begin{abstract}
The backwards diffusion equation is one of the classical ill-posed inverse
problems, related to a wide range of applications,
and has been extensively studied over the last 50 years.
One of the first methods was that of {\it quasireversibility\/} whereby
the parabolic operator is replaced by a differential operator
for which the backwards problem
in time is well posed. This is in fact the direction we will take but will do
so with a nonlocal operator; an equation of fractional order in time
for which the backwards problem is known to be ``almost well posed.''

We shall look at various possible options and strategies but our
conclusion for the best of these will exploit the linearity of the problem
to break the inversion into distinct frequency bands and to use a
different fractional order for each.
The fractional exponents will be chosen using the discrepancy principle
under the assumption we have an estimate of the noise level in the data.
An analysis of the method is provided as are some illustrative numerical
examples.
\end{abstract}

\section{Introduction}

The setting is in a bounded, simply connected domain $\Omega$ with
smooth ($C^2$) boundary $\partial\Omega$.
$\mathbb{L}$ is a uniformly elliptic second order partial differential
 operator defined in $\Omega$
with sufficiently smooth coefficients and subject to boundary values
on $\partial\Omega$ that, for simplicity, we take to be of homogeneous
Dirichlet type so that the domain of  $\mathbb{L}$
can be taken to be $H^2(\Omega)\cap H^1_0(\Omega)$.
There is, however, a completely parallel situation if the boundary conditions
are of impedance type.

Thus we have
\begin{equation}\label{eqn:direct_prob}
\begin{aligned}
u_t - \mathbb{L} u &= 0,\quad (x,t)\in\Omega\times(0,T) \\
u(x,t) &= 0 \quad (x,t)\in\partial\Omega\times(0,T) \\
u(x,0) &= u_0 \quad x\in\Omega \\
\end{aligned}
\end{equation}
where $u_0$ is unknown and has to be determined from the final time value
\begin{equation}\label{eqn:final_data}
u(x,T) = \uT(x) \quad x\in\Omega 
\end{equation}
for some $T>0$ and a measured function $\uT(x)$ taken over the domain $\Omega$.

This problem is well-known, and easily shown, to be extremely ill-posed.
While often viewed as the ``backwards heat problem,'' it in fact arises
anytime a diffusion process has to be reversed and governs a wide variety
of applications.
Some of these would dictate an initial state not governed by a smooth
function $u_0$ but one with significant information residing in the
mid and high frequency bands.  The reversal from a final state might not
be through a strict time-process. An example here is the degradation
of an image by a blurring process; the backwards problem becomes one
of de-blurring.
Indeed, the solution to \eqref{eqn:direct_prob} with \eqref{eqn:final_data}
can be represented in the form 
$g(x) = \int_\Omega K(x-y;T)u_0(y)\,dy$.
This is a Fredholm integral equation of the first kind for the 
\revision{initial} 
\Margin{Report C (1)}
state
$u_0$ and its inversion corresponds to deblurring from a perturbation of
a Gaussian kernel.

Further recent application examples we wish to mention is identification of
airborne contaminants \cite{Akceliketal2005} and imaging with acoustic or elastic waves
in the presence of strong attenuation, which leads to a similar setting after
reformulation as a first order in time system
\cite{AlabauCannarsa2009,ChenTriggiani1989,Lasieckaetal2013}, arising, e.g., 
in  photoacoustic tomography \cite{KowSch12}.

Given its physical importance, 
\eqref{eqn:direct_prob} with \eqref{eqn:final_data} has received considerable
attention over the last sixty years and in the next section
we review some of these approaches as they will have relevance to the main
results of this paper.

The standard regularization technique to invert a compact operator
is to replace it by a ``nearby'' operator with a bounded inverse
and for the case of \eqref{eqn:direct_prob} this has been with either another
differential operator or what is in effect a truncated singular value
decomposition of the original.
Our approach will be rather different; we seek to replace
the parabolic equation
\eqref{eqn:direct_prob} with a fractional subdiffusion operator
whereby the time derivative $u_t$ becomes $D^\alpha_t u$ for some
$\alpha$, $0<\alpha<1$.
\begin{equation}\label{eqn:sub_diffusion}
D^\alpha_t - \mathbb{L} u = 0,\quad (x,t)\in\Omega\times(0,T).
\end{equation}
The rationale behind this lies in the fact that the
parabolic equation arises from a diffusion model based on a Markov process
in which the current state of the system is determined from only
the previous state.
The model based on \eqref{eqn:sub_diffusion} is non-Markovian
and the value of the current state depends on all previous states;
indeed these have to be retained in the solution of \eqref{eqn:sub_diffusion}.
Thus in contrast to the parabolic differential operator we are now using a
non-local operator.
This fact allows for a more transparent reversal in time.
At the solution level the exponential function inherent in
\eqref{eqn:direct_prob} is replaced by a Mittag Leffler function for the
subdiffusion operator and the decay of this for large argument is only linear.
Indeed, it has been shown, \cite{SakamotoYamamoto:2011},
that the backwards subdiffusion problem is only mildly ill-conditioned;
equivalent to a two derivative loss in space.
However, as we shall see there are several complexities involved and
the replacement as a regularizer cannot be done without some care.

On the other hand, the subdiffusion equation \eqref{eqn:sub_diffusion}
is itself of considerable importance in applications.
If $D^\alpha_t$ is represented by a single fractional exponent $\alpha$
then backwards inversion can be accomplished in a  straightforward way.
However, if there are multiple exponents involved, that is
$D^\alpha_t = \sum_{i=1}^M q_i D^{\alpha_i}_t$ or, more generally if
$D^\alpha_t$ represents a fully distributed fractional derivative then
the regularization techniques discussed in this paper are exactly the pattern
that would have to be followed.

\section{
Quasi-Reversibility and Random Walk models
}

In this section we provide some background on why we view regularization
of the backwards parabolic equation by replacing it by one of subdiffusion type
is in a sense natural from a physical motivation standpoint.

Dating from the late 1960s
the initial attack on the inverse diffusion problem 
\eqref{eqn:direct_prob}, \eqref{eqn:final_data}
was by the method of
quasi-reversibility whereby the parabolic operator was replaced by
a ``nearby'' differential operator for which the time reversal was well
posed and the approach was popularized in the book by Lattes and Lions,
\cite{LattesLions:1969}.
Some examples suggested were adding a term $\epsilon u_{tt}$ so that 
the equation 
\eqref{eqn:direct_prob} was of hyperbolic type or a  fourth order
operator term $\epsilon \mathbb{L}^2$ (thus converting the heat equation into
the beam equation with lower order terms).
The difficulty with both these perturbations is that the new operators require
either further initial or further boundary conditions that are not
transparently available.
It should also be noted that the idea of adding a small, artificial term
to a differential operator in order to improve the ill-conditioning
of a numerical scheme, such as adding {\it artificial viscosity\/}
to control the behaviour of shocks, is even older.

The quasi-reversibility approach by Showalter, \cite{Showalter:1974a,Showalter:1974b,Showalter:1976}
was to instead use the pseudoparabolic equation
\begin{equation}\label{eqn:pseudoparabolic}
(I - \epsilon \mathbb{L})u^{\epsilon}_t - \mathbb{L}u^{\epsilon} = 0,
\quad (x,t)\in\Omega\times(0,T) 
\end{equation}
which has a natural setting of 
$H^2(\Omega)\cap H^1_0(\Omega)$ and subject to the single initial
condition $u^\epsilon(x,0)=u_0(x)$.
There is an interesting history to this equation.
It occurs independently in numerous applications such as a
two-temperature theory of thermodynamics and flow in porous media
\cite{Barenblatt:1960,ColemanDuffinMizel:1965,ChenGurtin:1968}
and is known in the Russian literature as an equation
of Sobolev type.
Of course, in these applications the additional term $\epsilon \mathbb{L}\,u_t$
was part of the extended model and not added merely for a stabilizing effect.

The operator $B_\epsilon:= (I - \epsilon\mathbb{L})^{-1}\mathbb{L}\,$ is a
bounded operator on 
\revision{$H^2(\Omega)\cap H_0^1(\Omega)$} 
\Margin{Report C (2)} 
for $\epsilon>0$.
Thus the full group of operators $\exp(-tB_\epsilon)$ is easily defined
by the power series $\sum\frac{(-t)^n}{n!} (B_\epsilon)^n$ under conditions
on the resolvent $R(\lambda,\mathbb{L})$ which are satisfied by
any strongly elliptic operator.
Under such conditions,
$\exp(-tB_\epsilon)$ converges to $\exp(-t\mathbb{L})$
in the strong topology as $\epsilon\to 0$ and is the basis of Yosida's proof
of the Hille-Phillips-Yosida Theorem which 
shows the existence of semigroups of differential operators.
There are known error estimates on the rate of this convergence.
The quasi-reversibility step is to recover an approximation to
$u_0(x)$ by computing $u_0(x;\epsilon) = \exp(-T B_\epsilon)\uT(x)$.
One has to
select an $\epsilon>0$, depending on the expected noise level $\delta$
in $\uT$ and using $u_0(x;\epsilon)$ as the approximation to the initial $u_0$.
Thus replacing the heat equation by the pseudoparabolic equation
is a regularizing method for solving the backwards heat problem and
well-studied in the literature.

Of course, there were other approaches.
For example, a blending of the quasi-reversibility ideas with those of
logarithmic convexity led Showalter to suggest that retaining the heat equation
but introducing the {\it quasi-boundary value\/}
\begin{equation}\label{eqn:quasi-boundary}
\epsilon u(x,0) + u(x,T) = \uT(x) \quad x\in\Omega 
\end{equation}
in place of the final value \eqref{eqn:final_data},
gives superior reconstructions.
Several authors have followed this idea, for example,
\cite{ClarkOppenheimer:1994}.
A summary article on some of this earlier work can be found in
\cite{AmesClark_etal:1998} and a more comprehensive discussion in the book \cite{Isakov}.

Our approach in this paper will be 
\revision{to} 
\Margin{Report C (3)}
take a rather different regularizing
equation; one of fractional order in time giving rise to a subdiffusion model.
We give some background on this to show why it is feasible and,
in a sense to be defined, natural.

The starting point for fractional calculus is the Abel fractional integral
operator,
$ I_a^\alpha f(x) = \frac{1}{\Gamma(\alpha)}\int_a^x \frac{f(s)}{(x-s)^{1-\alpha}}\,ds$.
Then a fractional (time) derivative can be defined by either
${}^R_a D^\alpha_t f  = \frac{d\ }{dt} I_a^\alpha f$ or by
${}^C_a D^\alpha_t f  = I_a^\alpha \frac{df}{ds}$.
The former is the Riemann-Liouville derivative of order $\alpha$
and the latter is the Djrbashyan-Caputo derivative.
Note that these are nonlocal operators and have a definite starting point $a$.

While obviously related, there are clear and important differences.
The Riemann-Liouville version allows definition in a wider class of
function spaces and this is important from an analysis perspective.
One disadvantage from a differential equations viewpoint is
how initial conditions should be interpreted. For example, the
R-L derivative of a constant is nonzero; in fact it is unbounded at the
origin $a$.
The Djrbashyan-Caputo version has no such drawback and is thus more
frequently seen in applications involving initial or boundary data.
It is the one we will take in this paper.
Since we will use $t=0$ as the initial point throughout, we will
simply write $\partial_t^\alpha$ to denote ${}^C_0 D^\alpha_t$.

This fractional derivative was studied extensively by the Armenian
mathematician M~.~M.~Djrbashyan, in his 1966 book (in Russian);
an English translation of this appeared in 1993, \cite{Djrbashian:1993}.
However, there was a considerable amount of earlier work on
the topic but only available in the Russian literature.
The geophysicist Michele Caputo rediscovered this version of the
fractional derivative in (1967), \cite{Caputo:1967},
as a tool for understanding seismological 
\revision{phenomena}, 
\Margin{Report C (4)}
and later with Francesco Mainardi in viscoelasticity where the memory
effect of these derivatives was crucial,
\cite{CaputoMainardi:1971a}.

In addition to sharing the same initial/boundary conditions, the
fractional diffusion equation
\begin{equation}\label{eqn:subdiffusion}
\partial_t^\alpha u - \mathbb{L} u = 0,\quad (x,t)\in\Omega\times(0,T)
\end{equation}
has additional connections with the parabolic operator
$u_t -  \mathbb{L} u$ which will be useful for subsequent understanding
and which we now describe below. 

The heat equation can be viewed as the macroscopic limit of the basic
continuous time random walk ({\sc ctrw}) process where
after each time step $\delta t$ a random direction is chosen
and the walker moves in that direction a length $\delta x$.
If $\delta t$, $\delta x \to 0$ such that the ratio 
\revision{$K = (\delta x)^2/\delta t$} 
\Margin{Report C (5)}
is held constant, then it is easily seen that the heat equation
$u_t - K\triangle u = 0$ ensues. The value of $K$, the diffusion constant,
couples the space and time scales.
In the more general situation, one assumes that the temporal and
spatial increments
$\Delta t_n = t_n - t_{n-1}$ and $\Delta x_n = x_n-x_{n-1}$
are independent, identically distributed random variables,
following probability density functions $\psi(t)$ and $\lambda (x)$,
respectively, which are
the waiting time and jump length distributions, respectively.
Thus the probability of $\Delta t_n$ lying in an interval $[a,b]$ is
$\,P(a<\Delta t_n<b) = \int_a^b \psi(t)\,dt\,$  and the probability of
$\Delta x_n$ lying in an interval $[c,d]$ is
$\,P(c<\Delta x_n <d) = \int_c^d \lambda(x)\,dx$.

Different types of {\sc ctrw} processes can be categorized by the
{\it characteristic waiting time\/}
$\tau =: E[\Delta t_n] = \int_0^\infty t\psi(t)\,dt$
and the {\it jump length variance\/}
$\Sigma^2 =: E[(\Delta x_n)^2] = \int_{-\infty}^\infty x^2\lambda (x)\,dx$.
being finite or diverging.
If both are finite then it can be shown, 
\cite{MontrollWeiss:1965},
that the {\sc ctrw} framework recovers the classical diffusion
equation, as long as the waiting time {\sc pdf} $\psi(t)$ has a finite mean
and the jump length {\sc pdf} $\lambda(x)$ has finite first and second moments.
Thus this more general setting case reduces to the basic Gaussian
process described by the fundamental solution of the heat equation.
This is a realization of the {\it Central Limit Theorem.}

On the other hand,
if the mean waiting time $\tau$ is infinite which could occur, for example,
when the particle might be trapped in a certain potential well,
then we could, for example, have a waiting time {\sc pdf}
 with the asymptotic behavior
$\psi(t) \sim \frac{A}{t^{1+\alpha}}$ as $t\to\infty$,
for some $\alpha\in(0,1)$, and $A>0$.
The (asymptotic) power law decay is heavy tailed
and allows occasional very large waiting time between consecutive walks.
The closer $\alpha$ is to zero, the slower is the decay and more likely
a long waiting time will take place.

It turns out that this changes the dynamics of the stochastic process
completely. 
Assuming a fixed spatial step size $\Delta x$,
the combined {\sc pdf} $p(x,t)$ in the physical domain is now given by
\begin{equation}\label{eqn:frac_func_sol}
  p(x,t) = \frac{1}{2\sqrt{K_\alpha t^\alpha}}M_{\alpha/2}\left(\frac{|x|}{\sqrt{K_\alpha t^\alpha}}\right),
\end{equation}
\revision{where}
\Margin{Report C (6)}
$M_\alpha(z)$ is a particular version of the Wright function
to be described in the next section and the diffusion coefficient $K_\alpha$
is again a coupling between the spatial and temporal scales.
Taking $\alpha\to 1$ in \eqref{eqn:frac_func_sol} recovers the
fundamental solution of the heat equation.
It also can be shown that \eqref{eqn:frac_func_sol} is the fundamental solution
of the subdiffusion operator \eqref{eqn:subdiffusion}.
This ties in the fact that the fractional diffusion equation
\eqref{eqn:subdiffusion} results from a {\sc ctrw} with a temporal {\sc pdf}
given by the above asymptotic behaviour.
For classical (Brownian) motion the mean square deviation of the particle
from its starting point obeys the relation
$\langle x^2\rangle \propto t$ whereas the subdiffusion model gives
$\langle x^2\rangle \propto t^\alpha$.
For some direct applications involving the subdiffusion process see, 
for example, \cite{SokolovKlafterBlumen:2002}.

In the above analysis we could have assumed a finite waiting time
but dropped the assumption of finite variance on the {\sc pdf} $\lambda$.
This indeed leads to a fractional derivative in space but we shall not use this
approach.
There are alternative ways to define a space fractional derivative
that will better suit our regularization purpose and we 
will briefly introduce one standard approach in
Section~\ref{sec:space_frac_derivatives}
as it will have relevance to our analysis of regularization operators.

As well as the above tie in between the parabolic and subdiffusion equations
there are some fundamental differences.
The most important of these from our current perspective is the fact
that the decay of the solution of \eqref{eqn:direct_prob} is exponential in time
leading to severe ill-conditioning of the backwards problem.
On the other hand that of \eqref{eqn:subdiffusion} is only linear decay in time
and resulting in the backwards problem being only very mildly ill-conditioned.
There are many caveats and details that must be resolved but the basic
principle behind the regularization of the backwards parabolic
by the backwards fractional diffusion equation relies on this key fact.

The important point we wish to stress is the fact that considering
fractional order operators is natural in the sense that they also arise from
similar random walk processes just as in the parabolic case.
Their distinguishing features give rise to exactly the properties
that we desire in our regularizing equation:
the nonlocal fractional operator ``stores'' all previous time values
and this history record enables a feasible backwards in time reconstruction.

We should point out that in the discussion of random walks we needed to
correlate the space and time scales through a {\it diffusion constant\/} $K$.  
This is incorporated into the leading coefficients of the operator
$\mathbb{L}$, but if we had $\mathbb{L} = \triangle$ then an explicit
$K$ would have to be brought in through $K\triangle$.
The units of $K$ are distance$^2/$time and even for excellent conductors
such as metals this is typically quite small, of the order of $10^{-5}$.
Thus in our scaling of \ref{eqn:direct_prob} we should consider the presence
of the coupled values $KT$.  By scaling $K$ to unity we are in fact
scaling the values of the final time $T$.  Thus values of $T$ in the paper
of the order of $10^{-2}\,-\,10^{-3}$ actually represent fairly long
waiting times.

\smallskip

In the next section we shall provide some necessary background information
on the key special functions needed for the subdiffusion operator;
those of Mittag-Leffler and of Wright.
We will also look at fractional powers of elliptic operators.
Finally, some background information on regularization methods, in
particular the discrepancy principle for choosing regularization parameters, 
is provided.

Section~\ref{sec:regularization_strategies} will describe various
regularization strategies for the recovery of $u_0$ in
\eqref{eqn:direct_prob} subject to 
\eqref{eqn:final_data} that rely on fractional derivatives and Section~\ref{sec:conv} provides a convergence analysis for some of them.
\revision{Section~\ref{sec:reconstructions}} 
\Margin{Report C (7)}
will show some numerical reconstructions based on the
algorithms presented in Section~\ref{sec:regularization_strategies}.

\section{Fractional operators and regularization parameter choice}\label{sec:frac+reg}

Here we collect background material to be used in the following sections.
First we describe the main function of fractional calculus,
the Mittag-Leffler function, as well as the Wright function needed
for a description of the fundamental solution of the fractional
subdiffusion operator.
Second, we introduce some notation for fractional derivatives and collect
a few basic lemmas that will be needed for our analysis and to obtain a
representation theorem for the subdiffusion operator which will be the core
of our regularization methods.

\subsection{The Mittag-Leffler function $\revision{E_{\alpha,1}}(-x)$}\label{sec:mlf}

An essential component of fractional derivative formulations is
the two-parameter Mittag-Leffler function $E_{\alpha,\beta}(z)$ defined by
\begin{equation}\label{eqn:mlf}
  E_{\alpha,\beta}(z) = \sum_{k=0}^\infty \frac{z^k}{\Gamma(\alpha k+\beta)}\qquad  \alpha>0, \ \beta\in\mathbb{R},\quad z\in \mathbb{C},
\end{equation}
This generalizes the exponential function ubiquitous to classical diffusion;
$E_{1,1}(z) = e^z$.

The following lemmata will be needed, these can be found in many sources
including \cite{Djrbashian:1993,MR3244285,JinRundell:2015}.
\begin{lemma}\label{lem:ML_recurrence}
For $0< \alpha\leq 1$ and $x>0$, $\lambda>0$
\begin{equation}\label{eqn:ML_recur1}
\alpha\, \lambda \frac{d\ }{dx} E_{\alpha,1}(-\lambda x) = -E_{\alpha,\alpha}(-\lambda x)
\end{equation}
For 
\revision{$\alpha>0$ and $\beta>1$}
\Margin{Report C (9)}
and $\lambda$ real
\begin{equation}\label{eqn:ML_recur2}
\frac{d\ }{dx}\, x^{\beta-1} E_{\alpha,\beta}(\lambda x^\alpha)
= x^{\beta-2} E_{\alpha,\beta-1}(\lambda x^\alpha)
\end{equation}
For 
\revision{$\alpha>0$ and $\beta>0$} 
\Margin{Report C (9)}
and $a$ real
\begin{equation}\label{eqn:ML_recur3}
\frac{d\ }{d
\revision{x}
} E_{\alpha,\beta}(a 
\revision{x}
) =
\frac{a}{\alpha 
\revision{x}
}\bigl(
E_{\alpha,\beta-1}(a 
\revision{x}
) - (\beta-1)E_{\alpha,\beta}(a 
\revision{x}
) \bigr)
\end{equation}
\end{lemma}

\begin{lemma}\label{lem:mlf-asymptotic}
Let $\alpha\in(0,1]$, $\beta\in\mathbb{R}$, 
\revision{$x\geq 0$}, 
\Margin{Report C (10)}
and $N\in\mathbb{N}$.
Then with 
\revision{$x\to\infty$}, 
\Margin{Report C (10)}
\begin{equation}\label{eqn:mlf-asymp}
E_{\alpha,\beta}(-
\revision{x}
) = \sum_{k=1}^N\frac{(-1)^{k-1}}{\Gamma(\beta-\alpha k)}
\frac{1}{
\revision{x}
^k} + O\Bigl(\frac{1}{
\revision{x}
^{N+1}}\Bigr).
\end{equation}
\end{lemma}

\begin{lemma}\label{lem:mlf-asympt-bound}
For every $\alpha\in(0,1)$, the uniform estimate
\begin{equation*}
  \frac{1}{1+\Gamma(1-\alpha)x}\leq \revision{E_{\alpha,1}}(-x)\leq \frac{1}{1+\Gamma(1+\alpha)^{-1}x}
\end{equation*}
holds over $\mathbb{R}^+$, where the bounding constants are optimal.
\end{lemma}

From Lemma~\ref{lem:mlf-asympt-bound} 
we obtain the stability estimate
\begin{lemma}\label{lem:mlf-stability_est}
\begin{equation}\label{eqn:stability_est}
\frac{1}{\lambda \, E_{\alpha,1}(-\lambda T^\alpha)}\leq 
\frac{1}{\lambda}+\Gamma(1-\alpha) T^\alpha\,
\leq \bar{C}\frac{1}{1-\alpha};
\end{equation}
for $\bar{C}=\frac{1}{4\lambda_1}+\frac{\sqrt{2}\pi}{4} \max\{T^{\frac34},T\}$ and all $\lambda\geq \lambda_1$, $\alpha\in[\frac34,1)$.
\end{lemma}
\revision{ 
\begin{proof}
The first inequality in \eqref{eqn:stability_est} is an immediate consequence of the lower bound in Lemma~\ref{lem:mlf-asympt-bound}. The second inequality can be obtained by using the reflection formula
$\,\Gamma(z)\Gamma(1-z)=\frac{\pi}{\sin(\pi z)}$, which allows to estimate 
\[
\Gamma(1-\alpha)
=\frac{1}{\Gamma(\alpha)}\frac{\pi(1-\alpha)}{\sin(\pi (1-\alpha))}\,\frac{1}{1-\alpha}
\leq \frac{1}{\Gamma(1)}\frac{\tfrac{\pi}{4}}{\sin(\tfrac{\pi}{4})}\,\frac{1}{1-\alpha}
\]
for $\alpha\in[\frac34,1)$.
\end{proof}
}
\Margin{Report C (11)}

This lemma when taken together with filtering of the data 
with some function $f_\gamma$ so that $\lambda f_\gamma(\lambda)\leq C_\gamma$ 
implies a bound on the noise propagation in time fractional reconstruction
of $u_0$.
The fact that the noise amplification grows only linearly with
$\frac{1}{1-\alpha}$ as $\alpha\to1$ is one of the key facts that renders
fractional backwards diffusion an attractive regularizing method.

Convergence of $E_{\alpha,1}(-\lambda T^\alpha)$ to $\exp(-\lambda T)$ as
$\alpha\to1$ is clear,
but to prove convergence of the backwards subdiffusion regularization $u_0^\delta(\cdot,\alpha)$ to $u_0$ we require rate estimates in terms of $1-\alpha$
\begin{lemma}\label{lem:rateE}
For any $\alpha_0\in(0,1)$ and $p\in[1,\frac{1}{1-\alpha_0})$, there exists $C=C(\alpha_0,p)>0$ such that for all $\lambda\geq \lambda_1$, $\alpha\in[\alpha_0,1)$
\begin{equation}\label{eqn:rateE}
\left|E_{\alpha,1}(-\lambda T^\alpha)-\exp(-\lambda T)\right|\leq C\lambda^{1/p}(1-\alpha)\,.
\end{equation}
\end{lemma}

\begin{proof}
To prove \eqref{eqn:rateE}, we employ an energy estimate for the ODE satisfied by $v(t):=E_{\alpha,1}(-\lambda t^\alpha)-\exp(-\lambda t)=u_{\alpha,\lambda}(t)-u_{1,\lambda}(t)$,
see Lemma~\ref{lem:frac_relaxation} below,
\[
\partial_t v+\lambda v=-(\partial_t^\alpha-\partial_t)u_{\alpha,\lambda}=:w\,.
\]
Multiplying with $|v(\tau)|^{p-1}\mbox{sign}(v(\tau))$, integrating from $0$ to $t$, and applying Young's inequality yields
\[
\begin{aligned}
\frac{1}{p}|v(t)|^p +\lambda\int_0^t |v(\tau)|^p\, d\tau 
&= \int_0^t w(\tau) |v(\tau)|^{p-1}\mbox{sign}(v(\tau))\, d\tau \\
&\leq \frac{1}{p\lambda^{p-1}}\int_0^t |w(\tau)|^p\, d\tau+\frac{(p-1)\lambda}{p}\int_0^t |v(\tau)|^p\, d\tau
\end{aligned}
\]
i.e., after multiplication with $p$,
\begin{equation}\label{vestp}
|v(t)|^p +\lambda\int_0^t |v(\tau)|^p\, d\tau \leq \frac{1}{\lambda^{p-1}}\int_0^t |w(\tau)|^p\, d\tau\,.
\end{equation}
We proceed by deriving 
\revision{an} 
\Margin{Report C (12)}
estimate of the the $L^p$ norm of $w$ of the form
\[
\left(\int_0^t |w(\tau)|^p\, d\tau\right)^{1/p}\leq C\, \lambda\, (1-\alpha) 
\]
with $C>0$ independent of $\alpha$ and $\lambda$.

We do so using its Laplace transform, and the fact that
\[
w=-(\partial_t^\alpha-\partial_t)u_{\alpha,\lambda}=\partial_tE_{\alpha,1}(-\lambda t^\alpha)-h_\alpha*\Bigl(\partial_tE_{\alpha,1}(-\lambda t^\alpha)\Bigr)
\]
where 
\[
h_\alpha(t)=\frac{1}{\Gamma(1-\alpha)} t^{-\alpha} \mbox{ with } (\mathcal{L}h)(\xi)=:H(\xi)=\xi^{\alpha-1}\,,
\]
due to $\Bigl(\mathcal{L}(t^p)\Bigr)(\xi)=\Gamma(1+p)\xi^{-(1+p)}$ for $p>-1$.
Using the identity 
\[
\Bigl(\mathcal{L}\Bigl(\partial_tE_{\alpha,1}(-\lambda t^\alpha)\Bigr)\Bigr)(\xi)=-\frac{\lambda}{\lambda+\xi^\alpha}
\]
(that follows from Lemma~\ref{lem:frac_relaxation} below)
together with the Convolution Theorem, we have, for any $\rho>0$ (fixed, independently of $\alpha\in[\alpha_0,1)$, e.g., $\alpha_0=\frac34$, $\rho=\frac18$),
\[
\Bigl(\mathcal{L}w\Bigr)(\xi)=\lambda \frac{\xi^{\alpha-1}-1}{\lambda+\xi^\alpha}
=\lambda\, A(\xi;\alpha)\, B(\xi,\alpha)\,.
\] 
Here 
\[
\begin{aligned}
&A(\xi;\alpha)=\frac{\xi^{\rho}}{\lambda+\xi^\alpha} =\mathcal{L}(a(t;\alpha))\,, \\
&B(\xi;\alpha)=\xi^{\alpha-1-\rho}-\xi^{-\rho} =\mathcal{L}(b(t;\alpha))=\mathcal{L}(\varphi(t;\alpha)-\varphi(t;1))\,,
\end{aligned}
\]
with 
\[
\begin{aligned}
\|a(\cdot;\alpha)\|_{L^q(0,T)}
\leq C \|A(\cdot;\alpha)\|_{L^{q^*}(\mathbb{R})}
\leq C_1<\infty 
\end{aligned}
\]
provided
\begin{equation}\label{q}
q^*> \frac{1}{\alpha-\rho} \mbox{ and }
\varphi(t,\alpha)=\frac{1}{\Gamma(1-\alpha+\rho)} t^{\rho-\alpha}\,.
\end{equation}
Hence, by the Mean Value Theorem and some $\tilde{\alpha}\in[\alpha,1]$ (note that $\tilde{\alpha}$ depends on $t$)
\[
\begin{aligned}
\varphi(t,\alpha)-\varphi(t,1)&=
\frac{d\varphi}{d\alpha} (t,\tilde{\alpha})\, (\alpha-1)\\
&=\frac{1}{\Gamma(1-\tilde{\alpha}+\rho)} t^{\rho-\tilde{\alpha}} \Bigl(-\frac{\Gamma'(1-\tilde{\alpha}+\rho)}{\Gamma(1-\tilde{\alpha}+\rho)} +\log(t)\Bigr)\, (1-\alpha)\,,
\end{aligned}
\]
so that 
\begin{equation}\label{r}
\|b(\cdot,1)\|_{L^r(0,T)}
=\|\varphi(\cdot,\alpha)-\varphi(\cdot,1)\|_{L^r(0,T)}\,
\leq C_2 (1-\alpha)
\mbox{ provided }
r<\frac{1}{1-\rho}\,.
\end{equation}
Altogether we have, 
\[
\begin{aligned}
\|w\|_{L^p(0,T)}&=
\lambda\|a(\cdot,\alpha)*b(\cdot;\alpha)\|_{L^p(0,T)}\\
&\leq
\lambda\|a(\cdot,\alpha)\|_{L^q(0,T)}
\|b(\cdot;\alpha)\|_{L^r(0,T)}\leq C_1 C_2 \lambda(1-\alpha)\,.
\end{aligned}\]
provided
$\frac{1}{q}+\frac{1}{r}\leq 1+\frac{1}{p}$.

Together with \eqref{q} and \eqref{r} and uniformity with respect to $\alpha\in [\alpha_0,1)$ this leads to the condition
$p<\frac{1}{1-\alpha_0}$.
\end{proof}

The above lemma together with the stability estimate \eqref{eqn:stability_est} yields the following bound which will be crucial for our convergence analysis in Section~\ref{sec:conv}.
\begin{lemma} \label{lem:boundH2}
For any $\alpha_0\in(0,1)$ and $p\in[1,\frac{1}{1-\alpha_0})$, there exists $\tilde{C}=\tilde{C}(\alpha_0,p)>0$ such that for all $\lambda\geq \lambda_1$, $\alpha\in[\alpha_0,1)$
\begin{equation}\label{eqn:boundH2}
\left|\frac{\exp(-\lambda T)}{E_{\alpha,1}(-\lambda T^\alpha)} - 1\right|\leq \tilde{C} \lambda^{1+1/p}\,.
\end{equation}
\end{lemma}

\subsection{The Wright function}\label{sec:wright}

For $\mu,\ \rho \in \mathbb{R}$ with $\rho>-1$, the Wright function $W_{\rho,\mu}(z)$, \cite{Wright:1933}, is defined by
\begin{equation}\label{eqn:Wright}
W_{\rho,\mu}(z) = \sum_{k=0}^\infty \frac{z^k}{k!\Gamma(\rho k + \mu)}\quad z \in \mathbb{C}.
\end{equation}
For any $\rho>-1$, $\mu\in\mathbb{R}$, the Wright function
$E_{\rho,\mu}(z)$ is entire of order $1/(1+\rho)$.

The reason for the importance of this function in subdiffusion lies in
the fact that the Laplace transform of a Wright function is a Mittag-Leffler
function
\begin{equation}\label{Laplace-W-ML}
\mathcal{L}[W_{\rho,\mu}(x)](z) = z^{-1}E_{\rho,\mu}(z^{-1}).
\end{equation}
Of course this is really
used in reverse to obtain the inverse Laplace transform
of a certain Mittag-Leffler function.

One case of the Wright function relevant to fractional diffusion is the
following $M$-Wright function, \cite{Mainardi:1996}
\begin{equation}\label{eqn:M-Wright}
M_{\mu}(z) = W_{-\mu,1-\mu}(-z)=\sum_{k=0}^\infty \frac{(-1)^kz^k}{k!\Gamma(1-\mu(k+1))}
\end{equation}

\begin{lemma}\label{thm:MWright-Fourier}
For $\mu\in(0,1)$, the Fourier transform of $M_\mu(|x|)$ is given by
\begin{equation*}
  \mathcal{F}[M_\mu(|x|)](\xi) = 2E_{2\mu}(-\xi^2).
\end{equation*}
\end{lemma}

\subsection{Solution of the subdiffusion equation}\label{sec:subdiffusion_sol}

Combining the lemmas in sections \ref{sec:mlf} and \ref{sec:wright}
we obtain the following results

\begin{lemma}\label{lem:frac_relaxation}
The initial value problem for the
fractional ordinary differential equation
$\partial_t^\alpha  u + \lambda u = 0$ for  $x>0$ and $0<\alpha <1$
with $u(0)=1$, has solution $u(t)$ given by
\begin{equation}\label{eqn:frac_relaxation}
  u(t) = \revision{E_{\alpha,1}}(-\lambda t^\alpha) = \sum_{k=0}^\infty \frac{(-\lambda t^\alpha)^k}{\Gamma(k\alpha +1)}.
\end{equation}
The solution satisfies
\begin{equation}\label{eqn:frac_relaxation_sol}
  \revision{E_{\alpha,1}}(-\lambda x^\alpha) >0 \quad \mbox{and} \quad\frac{d}{dx} \revision{E_{\alpha,1}}(-\lambda x^\alpha)<0\quad \forall x>0.
\end{equation}
\end{lemma}

From the above we easily obtain by separation of variables
and using the eigenvalues and -functions $\lambda_j\in\mathbb{R}^+$, $\phi_j\in H^2(\Omega)\cap H_0^1(\Omega)\,$ of $\,-\mathbb{L}$
\begin{lemma}\label{lem:subdiffusion_sol}
The solution of \eqref{eqn:subdiffusion} is given by
\begin{equation}\label{eqn:subdiffusion_sol}
u(x,t) = \sum_{n=0}^\infty \langle u_0,\phi_n\rangle
 \revision{E_{\alpha,1}}(-\lambda_n t^\alpha)\phi_n(x)
\end{equation}
\end{lemma}

By taking Fourier transforms in space and Laplace in time
using the above lemmas give,
\cite{Mainardi:1996}
\begin{lemma}\label{lem:fundamental_sol}
The fundamental solution $G_\alpha(x,t)$ is given by
\begin{equation}\label{eqn:subdiff_fund}
G_\alpha(x,t) = \frac{1}{2t^\frac{\alpha}{2}}M_\frac{\alpha}{2}\Bigl(\frac{|x|}{\sqrt{t^\alpha}}\Bigr).
\end{equation}
Note that for $\alpha\in(0,1)$, for every $t>0$,
the function $x\to G_\alpha(x,t)$ is not
differentiable at $x=0$; in fact it fails to be Lipschitz at $x=0$.
\end{lemma}

The limited smoothness of the fundamental solution results in
limited smoothness of the subdiffusion equation.
The following result, \cite{SakamotoYamamoto:2011}, is critical

\begin{lemma}\label{lem:subdiff_smoothness}
Let $0<\alpha<1$ and $u_0\in L^2(\Omega)$.
Then there exists a unique weak solution
$u \in C([0, T ]; L^2(\Omega)) \cap C((0,T]; H^2(\Omega) \cap H^1_0(\Omega)
\revision{)}
$ 
\Margin{Report C (13)}
to \eqref{eqn:subdiffusion} with $u(x,0)=u_0(x)$ such that
$\partial_t^\alpha u \in C((0,T]; L^2(\Omega))$
and a constant $C > 0$ such that
\begin{equation}\label{eqn:subdiff_reg}
\|u(\,\cdot\,,t)\|_{H^2} + \|\partial_t^\alpha u\|_{L^2}
\leq C t^{-\alpha} \|u_0\|_{L^2}
\end{equation}
\end{lemma}

\subsection{Space fractional derivatives}\label{sec:space_frac_derivatives}

While one can use derivatives based on the Abel integral for space variables
there is also a considerable literature on fractional powers of operators.
For example, the Fourier transform of $-\triangle$ defined on $\real^n$ has
symbol $\xi^2$ and hence the fractional power of order $\beta$ of $-\triangle$
can be defined as the pseudodifferential operator whose symbol is $\xi^\beta$.
In the case of bounded domains $\Omega$ we can proceed as follows.

We define an operator $A$ in $L^2(\Omega)$ by
$\,(Au)(x) = 
\revision{(-\mathbb{L}} 
u)(x)$ 
\Margin{Report C (14)}
for $x\in\Omega$
with its domain $D(A)=H^2(\Omega)\cap H_0^1(\Omega)$.
Since
$A$ is a self-adjoint, uniformly elliptic operator, the spectrum of $A$ is
entirely composed of eigenvalues and counting according to the multiplicities,
we can set $0<\lambda_1\leq \lambda_2\ldots$.
By $\phi_j\in H^2(\Omega)\cap H_0^1(\Omega)$,
we denote the $L^2(\Omega)$ orthonormal eigenfunctions corresponding to
$\lambda_j$.
Then from \cite{Kato:1960}, the fractional power
$A^\beta$ is defined for any $\beta\in\mathbb{R}$ by
\begin{equation}\label{eqn:frac_power_op}
A^\beta f = \sum_1^\infty \lambda_j^\beta\langle f,\phi_j\rangle \phi_j
\end{equation}
Next we introduce a space $\dH s$ by
\begin{equation*}
  \dH s = \Bigl\{v\in L^2(\Omega): \sum_{j=1}^\infty \lambda_j^{s}|(v,\varphi_j)|^2<\infty\Bigr\}
\end{equation*}
and that $\dH s$ is a Hilbert space with the norm
$  \|v\|_{\dH s}^2=\sum_{j=1}^\infty \lambda_j^{s}|(v,\varphi_j)|^2$.
By definition, we have the following equivalent form:
\begin{equation*}
  \|v\|_{\dH s}^2 = \sum_{j=1}^\infty |(v,\lambda_j^\frac{s}{2}\varphi_j)|^2 = \sum_{j=1}^\infty (v,A^\frac{s}{2}\varphi_j)^2=\sum_{j=1}^\infty (A^\frac{s}{2}v,\varphi_j)^2=\|A^\frac{s}{2}v\|_{L^2(\Omega)}^2.
\end{equation*}
We have $\dH s\subset H^{s}(\Omega)$ for $s>0$.
In particular,
$\dH 1 =H_0^1(\Omega)$.
Since $\dH s \subset L^2(\Omega)$, identifying
the dual $(L^2(\Omega))^\prime$ with itself, we have $\dH s \subset L^2(\Omega)
\subset (\dH s)^\prime$.
Henceforth, we set $\dH {-s}=(\dH s)^\prime$,
which consists of bounded linear functionals on $\dH s$.

The standard pseudoparabolic equation can 
\revision{be}
\Margin{Report C (16)}
generalized to elliptic operators
$A$ and $B$ not necessarily of the same order.
Continuing with same structure, assuming that both $A$ and $B$ are positive
operators in the sense that
$\langle Ax,x\rangle > 0$, $\langle Bx,x\rangle > 0$ we form the equation
\begin{equation*}
(I +\epsilon A)u_t + B u = 0
\end{equation*}
If both operators are of the same order we have a straightforward perturbation
of the pseudoparabolic equation.
If the order of $A$ is greater than that of $B$ then 
$C_\epsilon := (I +\epsilon A)^{-1}B$ will be bounded (in fact compact)
on $L^2$ and a full group $e^{-tC_\epsilon}$ will result.
Conversely, if the order of $B$  is greater than that of $A$
then  $C_\epsilon$ is unbounded and we will obtain a semigroup
once again from $C_\epsilon$ in the limit as $\epsilon\to 0$.

Our interest here is in the case that $B=-\mathbb{L}$ and 
$A = (-\mathbb{L})^\beta$; that is a fractional power of $-\mathbb{L}$.
\begin{equation}\label{eqn:beta-pseudoparabolic}
(I +\epsilon (-\mathbb{L})^\beta)u_t + \mathbb{L} u = 0
\end{equation}
This $\beta-$pseudoparabolic equation is no longer a regularizer for the
backwards parabolic equation $u_t + \mathbb{L}
\revision{u}
 = 0$ if $\beta<1$
\Margin{Report C (17)}
but we expect it to have partial regularizing properties and the exploration
of this will be studied in the next sections.

Finally, we can combine both space and time fractional derivatives to obtain
\begin{equation}\label{eqn:alpha-beta-pseudoparabolic}
(I +\epsilon (-\mathbb{L})^\beta)\partial_t^\alpha u + \mathbb{L} u = 0
\end{equation}

Define $\mu_n$ by $\mu_n = \frac{\lambda_n}{1 + \epsilon\lambda_n^\beta}$.
Then the solution to \eqref{eqn:beta-pseudoparabolic}
has the representation
\begin{equation}\label{eqn:beta-pseudoparabolic_sol}
u(x,t;\beta,\epsilon) = \sum_{n=1}^\infty \langle u_0,\phi_n\rangle
e^{-\mu_n T} \phi_n(x)
\end{equation}
while the solution to  equation \eqref{eqn:alpha-beta-pseudoparabolic} has 
the representation
\begin{equation}\label{eqn:beta-pseudoparabolic_frac_sol}
u(x,t;\alpha,\beta,\epsilon) = \sum_{n=1}^\infty \langle u_0,\phi_n\rangle
E_{\alpha,1}(-\mu_n\, t^\alpha) \phi_n(x)
\end{equation}

\subsection{The Morozov Discrepancy Principle}\label{sec:morozov}

As in every regularization method,
certain parameters have to be chosen appropriately as part of a trade-off
such that on one hand the ill-posed problem is stabilized,
but on the other the approximation error arising from the modification
of the problem by the addition of the  stabilizing terms
does not become too large. 
Regularization parameters appearing in the methods considered in this paper are,
for example: the fractional orders $\alpha$ and $\beta$ of the time
or space derivatives, respectively; the multiplier $\epsilon$ in these
 pseudoparabolic equation; and later in the paper, the indices $K_i$
at which we split the frequency band for treatment with different methods.

There exists a large body of literature on regularization parameter choices;
an overview on regularization parameter choice rules with many relevant
references can be found in \cite[Chapter 4]{EnglHankeNeubauer:1996},
\cite[Chapter 7]{Hansen1997}, and more recently, in 
\cite[Chapters 2,3]{LuPereverzyev2013}.

In this paper, we will follow a rather classical,
but also versatile, paradigm for regularization parameter choice,
namely the {\it discrepancy principle}.
This dates back to Morozov's well-known paper \cite{Morozov67}.  
The idea is to choose, out of a family of regularized problems,
the most stable one such that the residual is of the order of magnitude
of the expected noise level. 
In the context of, e.g. subdiffusion regularization
$u_0^\delta(x;\alpha)=u(x,0)$, where $u$ solves
\eqref{eqn:subdiffusion} with given noisy final data $\uTdel$,
the discrepancy principle requires one to choose $\alpha$ such that the
difference between the final data simulated from the reconstruction
$\exp(\mathbb{L}T)u_0^\delta(x;\alpha)$ differs from the noisy data
$\uTdel$ by not more than the noise level $\delta$,
while $\alpha$ is kept as far away as possibly from the critical value
$\hat{\alpha}=1$
\[
\alpha\in\mbox{argmin}\{\alpha'\, : \, \|\exp(\mathbb{L}T)u_0^\delta(\cdot,\alpha')-\uTdel\|_{L^2} \leq \delta\}\,,
\]
where $\delta$ is an estimate on the $L^2$ norm of the noise, cf. \eqref{eqn:delta}.
We will actually apply this in a relaxed, easier to compute manner,
and to a smoothed version of the data.

Of course a crucial point here is knowledge of the noise level
(or of a good estimate on it),
which is admittedly not available in some applications.
On the other, whenever $\delta$ is known, the discrepancy can often be
proven to yield a convergent regularization method,
even one with optimal convergence rates.

/bin/bash: a: command not found
\newdimen\xfiglen \newdimen\yfiglen
\xfiglen=4 true in
\yfiglen=2 true in
\newbox\figurelegend
\newbox\figurereconlegend
\newbox\figurereconlegendtwo
\newbox\figurereconlegendthree
\newbox\figureone
\newbox\figuretwo
\newbox\figurethree
\newbox\figurefour
\newbox\figurefive

\setbox\figurelegend=\hbox{
\small
\beginpicture
  \setcoordinatesystem units <0.9true in,0.7true in> point at 0 -0.7
  \setplotarea x from 0 to 0.8, y from 0 to 0.5
\footnotesize
\linethickness=0.7pt
\setdots <2pt>
  \putrule from 0 0.4 to 0.2 0.4  
  \OliveGreen{\relax
  \putrule from 0 0.4 to 0.2 0.4 }\relax
  \put {$\alpha = 0.5$} [l] at 0.3 0.4
  \Blue{\relax
   \putrule from 0 0.2 to 0.2 0.2 }\relax
\setdashes <3pt>
  \putrule from 0 0.2 to 0.2 0.2
  \put {$\alpha = 0.9$} [l] at 0.3 0.2
\setsolid
  \putrule from 0 0.0 to 0.2 0.0
  \put {$\alpha = 1$} [l] at 0.3 0.0
\endpicture
\relax
}
\setbox\figurereconlegend=\hbox{
\small
\beginpicture
  \setcoordinatesystem units <1true in,1true in> point at 0 -0.7
  \setplotarea x from 0 to 0.8, y from 0 to 0.7
\footnotesize
\linethickness=0.8pt
\put {$T=0.02$\ $\;\delta=0.001$} [l] at 0 0.6
\setsolid
  \putrule from 0 0.4 to 0.2 0.4
  \put {Actual $u_0$} [l] at 0.3 0.4
\setplotsymbol ({\sevenrm .})
\setdots <2pt>
  \putrule from 0 0.2 to 0.2 0.2  
  \OliveGreen{\relax
  \putrule from 0 0.2 to 0.2 0.2 }\relax
  \put {single split freq} [l] at 0.3 0.2 
  \Blue{\relax
   \putrule from 0 0.0 to 0.2 0.0 }\relax
\setdashes <3pt>
  \putrule from 0 0.0 to 0.2 0.0
  \put {double split freq} [l] at 0.3 0.0 
\endpicture
}
\setbox\figurereconlegendtwo=\hbox{
\small
\beginpicture
  \setcoordinatesystem units <1true in,0.8true in> point at 0 -0.7
  \setplotarea x from 0 to 0.8, y from 0 to 0.7
\footnotesize
\linethickness=0.8pt
\put {$T=0.02$\ $\;\delta=0.01$} [l] at 0 0.6
\setsolid
  \putrule from 0 0.4 to 0.2 0.4
  \put {Actual $u_0$} [l] at 0.3 0.4
\setdots <2pt>
  \putrule from 0 0.2 to 0.2 0.2  
  \Red{\relax
  \putrule from 0 0.2 to 0.2 0.2 }\relax
  \put {SVD} [l] at 0.3 0.2
  \Blue{\relax
   \putrule from 0 0.0 to 0.2 0.0 }\relax
\setdashes <3pt>
  \putrule from 0 0.0 to 0.2 0.0
  \put {double split freq} [l] at 0.3 0.0 
\endpicture
}
\setbox\figurereconlegendthree=\hbox{
\small
\beginpicture
  \setcoordinatesystem units <1true in,0.8true in> point at 0 -0.7
  \setplotarea x from -0.2 to 0.8, y from 0 to 0.9
\footnotesize
\linethickness=0.8pt
\put {$T=0.02$\ $\;\delta=0.001$} [l] at 0 0.8
\setsolid
  \putrule from -0.1 0.6 to 0.2 0.6
  \put {Actual $u_0$} [l] at 0.3 0.6
\setdots <2pt>
  \Red{\relax
  \putrule from -0.1 0.4 to 0.2 0.4 }\relax
  \put {SVD} [l] at 0.3 0.4 
\setdashes <4pt>
  \OliveGreen{\relax
  \putrule from -0.1 0.2 to 0.2 0.2 }\relax
  \put {double-split freq} [l] at 0.3 0.2
\setdashpattern <3pt,2pt,1pt,2pt> 
  \Blue{\relax
   \putrule from -0.1 0.0 to 0.2 0.0 }\relax
  \put {triple-split freq} [l] at 0.3 0.0
\endpicture
}
\setbox\figureone=\vbox{\hsize=\xfiglen
\beginpicture
\footnotesize
  \setcoordinatesystem units <0.00136\xfiglen,0.33\yfiglen> 
  \setplotarea x from 0 to 600, y from 0 to 3
  \axis bottom shiftedto y=0 ticks short numbered from 0 to 600 by 100 /
  \axis left ticks short withvalues ${10^0}$ ${10^1}$ ${10^2}$ ${10^3}$ / at 0 1 2 3 / /
\small
\put {\copy\figurelegend} [rb] at 600 0.2
\put {$\lambda$} [lb] at 600 0.1
\put {$\log_{10}(A)$} [l] at 10 3
\linethickness=0.8pt
\setplotsymbol ({\eightrm .})
\setquadratic
 \setdots <3pt>
 \OliveGreen{\relax
\plot
         0         0
   15.0000    0.4927
   30.0000    0.7471
   45.0000    0.9119
   60.0000    1.0326
   75.0000    1.1274
   90.0000    1.2055
  105.0000    1.2717
  120.0000    1.3293
  135.0000    1.3801
  150.0000    1.4256
  165.0000    1.4669
  180.0000    1.5045
  195.0000    1.5392
  210.0000    1.5713
  225.0000    1.6012
  240.0000    1.6292
  255.0000    1.6554
  270.0000    1.6802
  285.0000    1.7037
  300.0000    1.7259
  315.0000    1.7471
  330.0000    1.7673
  345.0000    1.7866
  360.0000    1.8050
  375.0000    1.8228
  390.0000    1.8398
  405.0000    1.8562
  420.0000    1.8719
  435.0000    1.8872
  450.0000    1.9019
  465.0000    1.9161
  480.0000    1.9299
  495.0000    1.9433
  510.0000    1.9562
  525.0000    1.9688
  540.0000    1.9810
  555.0000    1.9929
  570.0000    2.0045
  585.0000    2.0158
  600.0000    2.0268
 /\relax}\relax
%
\setplotsymbol ({\sevenrm .})
 \setdashes <4pt>
 \Blue{\relax
 \plot
         0         0
   15.0000    0.1059
   30.0000    0.2089
   45.0000    0.3088
   60.0000    0.4053
   75.0000    0.4982
   90.0000    0.5874
  105.0000    0.6727
  120.0000    0.7540
  135.0000    0.8313
  150.0000    0.9043
  165.0000    0.9733
  180.0000    1.0382
  195.0000    1.0990
  210.0000    1.1560
  225.0000    1.2092
  240.0000    1.2589
  255.0000    1.3053
  270.0000    1.3486
  285.0000    1.3890
  300.0000    1.4267
  315.0000    1.4620
  330.0000    1.4950
  345.0000    1.5259
  360.0000    1.5550
  375.0000    1.5824
  390.0000    1.6082
  405.0000    1.6326
  420.0000    1.6557
  435.0000    1.6777
  450.0000    1.6986
  465.0000    1.7185
  480.0000    1.7375
  495.0000    1.7557
  510.0000    1.7731
  525.0000    1.7899
  540.0000    1.8059
  555.0000    1.8214
  570.0000    1.8364
  585.0000    1.8508
  600.0000    1.8647
 /\relax}\relax
%
 \setsolid
\setlinear
 \Black{\relax
  \plot
         0         0
   60.0000   0.2606
  120.0000   0.5212
  180.0000   0.7817
  240.0000   1.0423
  300.0000   1.3029
  360.0000   1.5635
  420.0000   1.8240
  480.0000   2.0846
  540.0000   2.3452
  600.0000   2.6058
  /\relax}\relax
\endpicture
}
\xfiglen=3 true in
\yfiglen= 2.3 true in
\setbox\figuretwo=\vbox{\hsize=\xfiglen
\beginpicture
\footnotesize
  \setcoordinatesystem units <\xfiglen,0.5\yfiglen> 
  \setplotarea x from 0 to 1, y from 0 to 2
  \axis bottom shiftedto y=0 ticks short numbered from 0 to 1 by 0.2 /
  \axis left ticks short numbered from 0 to 2 by 0.5 /
\small
\put {\copy\figurereconlegend} [rt] at 0.95 2
\put {$x$} [lb] at 1 0.05
\put {$u_0(x)$} [lt] at 0.02 2
\setplotsymbol ({\sevenrm .})
\setquadratic
\setsolid
\Black{\relax 
\plot
         0         0
    0.0100    0.1051
    0.0200    0.2192
    0.0300    0.3402
    0.0400    0.4660
    0.0500    0.5944
    0.0600    0.7235
    0.0700    0.8511
    0.0800    0.9753
    0.0900    1.0942
    0.1000    1.2061
    0.1100    1.3096
    0.1200    1.4034
    0.1300    1.4862
    0.1400    1.5573
    0.1500    1.6159
    0.1600    1.6616
    0.1700    1.6942
    0.1800    1.7135
    0.1900    1.7197
    0.2000    1.7132
    0.2100    1.6945
    0.2200    1.6641
    0.2300    1.6230
    0.2400    1.5719
    0.2500    1.5120
    0.2600    1.4442
    0.2700    1.3697
    0.2800    1.2896
    0.2900    1.2052
    0.3000    1.1176
    0.3100    1.0281
    0.3200    0.9376
    0.3300    0.8474
    0.3400    0.7584
    0.3500    0.6716
    0.3600    0.5878
    0.3700    0.5080
    0.3800    0.4327
    0.3900    0.3625
    0.4000    0.2980
    0.4100    0.2395
    0.4200    0.1873
    0.4300    0.1416
    0.4400    0.1025
    0.4500    0.0699
    0.4600    0.0439
    0.4700    0.0241
    0.4800    0.0105
    0.4900    0.0026
    0.5000         0
    0.5100    0.0224
    0.5200    0.0493
    0.5300    0.0802
    0.5400    0.1145
    0.5500    0.1518
    0.5600    0.1915
    0.5700    0.2330
    0.5800    0.2759
    0.5900    0.3196
    0.6000    0.3635
    0.6100    0.4074
    0.6200    0.4506
    0.6300    0.4929
    0.6400    0.5338
    0.6500    0.5730
    0.6600    0.6104
    0.6700    0.6456
    0.6800    0.6784
    0.6900    0.7088
    0.7000    0.7366
    0.7100    0.7619
    0.7200    0.7845
    0.7300    0.8046
    0.7400    0.8222
    0.7500    0.8374
    0.7600    0.8103
    0.7700    0.7812
    0.7800    0.7502
    0.7900    0.7174
    0.8000    0.6832
    0.8100    0.6477
    0.8200    0.6112
    0.8300    0.5739
    0.8400    0.5361
    0.8500    0.4979
    0.8600    0.4596
    0.8700    0.4214
    0.8800    0.3835
    0.8900    0.3462
    0.9000    0.3094
    0.9100    0.2735
    0.9200    0.2385
    0.9300    0.2045
    0.9400    0.1716
    0.9500    0.1399
    0.9600    0.1095
    0.9700    0.0803
    0.9800    0.0523
    0.9900    0.0256
    1.0000         0
 /\relax}\relax
\setplotsymbol ({\ninerm .})
\setdots <3pt> 
\OliveGreen{\relax 
\plot
         0         0
    0.0100    0.1377
    0.0200    0.2751
    0.0300    0.4119
    0.0400    0.5471
    0.0500    0.6794
    0.0600    0.8072
    0.0700    0.9285
    0.0800    1.0419
    0.0900    1.1463
    0.1000    1.2412
    0.1100    1.3268
    0.1200    1.4040
    0.1300    1.4736
    0.1400    1.5362
    0.1500    1.5917
    0.1600    1.6391
    0.1700    1.6767
    0.1800    1.7019
    0.1900    1.7126
    0.2000    1.7071
    0.2100    1.6847
    0.2200    1.6464
    0.2300    1.5944
    0.2400    1.5317
    0.2500    1.4617
    0.2600    1.3874
    0.2700    1.3107
    0.2800    1.2325
    0.2900    1.1528
    0.3000    1.0709
    0.3100    0.9865
    0.3200    0.8999
    0.3300    0.8123
    0.3400    0.7262
    0.3500    0.6443
    0.3600    0.5693
    0.3700    0.5029
    0.3800    0.4458
    0.3900    0.3969
    0.4000    0.3542
    0.4100    0.3149
    0.4200    0.2768
    0.4300    0.2386
    0.4400    0.2001
    0.4500    0.1628
    0.4600    0.1288
    0.4700    0.1005
    0.4800    0.0799
    0.4900    0.0678
    0.5000    0.0641
    0.5100    0.0677
    0.5200    0.0770
    0.5300    0.0908
    0.5400    0.1082
    0.5500    0.1294
    0.5600    0.1552
    0.5700    0.1867
    0.5800    0.2244
    0.5900    0.2683
    0.6000    0.3173
    0.6100    0.3690
    0.6200    0.4210
    0.6300    0.4703
    0.6400    0.5147
    0.6500    0.5532
    0.6600    0.5855
    0.6700    0.6125
    0.6800    0.6356
    0.6900    0.6562
    0.7000    0.6754
    0.7100    0.6936
    0.7200    0.7106
    0.7300    0.7254
    0.7400    0.7372
    0.7500    0.7449
    0.7600    0.7480
    0.7700    0.7459
    0.7800    0.7389
    0.7900    0.7270
    0.8000    0.7107
    0.8100    0.6901
    0.8200    0.6656
    0.8300    0.6375
    0.8400    0.6063
    0.8500    0.5725
    0.8600    0.5366
    0.8700    0.4994
    0.8800    0.4614
    0.8900    0.4232
    0.9000    0.3848
    0.9100    0.3464
    0.9200    0.3078
    0.9300    0.2691
    0.9400    0.2303
    0.9500    0.1913
    0.9600    0.1524
    0.9700    0.1137
    0.9800    0.0755
    0.9900    0.0376
    1.0000    0.0000
 /\relax}\relax
 \setdashes <4pt> 
\Blue{\relax 
\plot
         0         0
    0.0100    0.1465
    0.0200    0.2913
    0.0300    0.4331
    0.0400    0.5702
    0.0500    0.7016
    0.0600    0.8260
    0.0700    0.9427
    0.0800    1.0510
    0.0900    1.1508
    0.1000    1.2417
    0.1100    1.3238
    0.1200    1.3971
    0.1300    1.4614
    0.1400    1.5167
    0.1500    1.5629
    0.1600    1.5998
    0.1700    1.6274
    0.1800    1.6456
    0.1900    1.6546
    0.2000    1.6546
    0.2100    1.6459
    0.2200    1.6285
    0.2300    1.6026
    0.2400    1.5680
    0.2500    1.5245
    0.2600    1.4717
    0.2700    1.4096
    0.2800    1.3384
    0.2900    1.2586
    0.3000    1.1713
    0.3100    1.0778
    0.3200    0.9799
    0.3300    0.8797
    0.3400    0.7789
    0.3500    0.6796
    0.3600    0.5834
    0.3700    0.4918
    0.3800    0.4061
    0.3900    0.3273
    0.4000    0.2565
    0.4100    0.1945
    0.4200    0.1419
    0.4300    0.0990
    0.4400    0.0660
    0.4500    0.0427
    0.4600    0.0284
    0.4700    0.0223
    0.4800    0.0235
    0.4900    0.0310
    0.5000    0.0437
    0.5100    0.0610
    0.5200    0.0823
    0.5300    0.1071
    0.5400    0.1352
    0.5500    0.1664
    0.5600    0.2004
    0.5700    0.2369
    0.5800    0.2757
    0.5900    0.3164
    0.6000    0.3584
    0.6100    0.4014
    0.6200    0.4450
    0.6300    0.4886
    0.6400    0.5318
    0.6500    0.5740
    0.6600    0.6148
    0.6700    0.6532
    0.6800    0.6887
    0.6900    0.7202
    0.7000    0.7469
    0.7100    0.7681
    0.7200    0.7832
    0.7300    0.7916
    0.7400    0.7933
    0.7500    0.7883
    0.7600    0.7768
    0.7700    0.7595
    0.7800    0.7369
    0.7900    0.7099
    0.8000    0.6794
    0.8100    0.6461
    0.8200    0.6110
    0.8300    0.5746
    0.8400    0.5376
    0.8500    0.5005
    0.8600    0.4637
    0.8700    0.4273
    0.8800    0.3917
    0.8900    0.3568
    0.9000    0.3228
    0.9100    0.2896
    0.9200    0.2570
    0.9300    0.2249
    0.9400    0.1932
    0.9500    0.1615
    0.9600    0.1297
    0.9700    0.0977
    0.9800    0.0654
    0.9900    0.0328
    1.0000    0.0000
 /\relax}\relax
\endpicture
}
\setbox\figurethree=\vbox{\hsize=\xfiglen
\beginpicture
\footnotesize
  \setcoordinatesystem units <\xfiglen,0.15\yfiglen> 
  \setplotarea x from 0 to 1, y from -2 to 4.5
  \axis bottom shiftedto y=0 ticks short numbered from 0 to 1 by 0.2 /
  \axis left ticks short numbered from -2 to 4 by 1 /
\small
\put {\copy\figurereconlegendtwo} [rt] at 1.02 4.5
\put {$x$} [lb] at 1 0.05
\put {$u_0(x)$} [lb] at 0.01 4.4
\setquadratic
\setsolid
\setplotsymbol ({\eightrm .})
\Black{\relax 
\plot
        0         0
    0.0100    1.4702
    0.0200    2.7518
    0.0300    3.6928
    0.0400    4.2064
    0.0500    4.2852
    0.0600    3.9986
    0.0700    3.4743
    0.0800    2.8681
    0.0900    2.3302
    0.1000    1.9744
    0.1100    1.8572
    0.1200    1.9709
    0.1300    2.2504
    0.1400    2.5927
    0.1500    2.8824
    0.1600    3.0190
    0.1700    2.9385
    0.1800    2.6275
    0.1900    2.1253
    0.2000    1.5158
    0.2100    0.9105
    0.2200    0.4263
    0.2300    0.1630
    0.2400    0.1833
    0.2500    0.4994
    0.2600    1.0692
    0.2700    1.8011
    0.2800    2.5689
    0.2900    3.2329
    0.3000    3.6644
    0.3100    3.7696
    0.3200    3.5077
    0.3300    2.8998
    0.3400    2.0265
    0.3500    1.0139
    0.3600    0.0101
    0.3700   -0.8428
    0.3800   -1.4359
    0.3900   -1.7137
    0.4000   -1.6833
    0.4100   -1.4095
    0.4200   -0.9977
    0.4300   -0.5680
    0.4400   -0.2270
    0.4500   -0.0432
    0.4600   -0.0327
    0.4700   -0.1581
    0.4800   -0.3397
    0.4900   -0.4780
    0.5000   -0.4800
    0.5100   -0.2835
    0.5200    0.1264
    0.5300    0.7119
    0.5400    1.3908
    0.5500    2.0529
    0.5600    2.5844
    0.5700    2.8902
    0.5800    2.9127
    0.5900    2.6414
    0.6000    2.1124
    0.6100    1.3995
    0.6200    0.5986
    0.6300   -0.1893
    0.6400   -0.8751
    0.6500   -1.3931
    0.6600   -1.7078
    0.6700   -1.8149
    0.6800   -1.7369
    0.6900   -1.5160
    0.7000   -1.2050
    0.7100   -0.8577
    0.7200   -0.5207
    0.7300   -0.2276
    0.7400    0.0054
    0.7500    0.1794
    0.7600    0.3101
    0.7700    0.4218
    0.7800    0.5398
    0.7900    0.6820
    0.8000    0.8523
    0.8100    1.0365
    0.8200    1.2039
    0.8300    1.3129
    0.8400    1.3205
    0.8500    1.1939
    0.8600    0.9202
    0.8700    0.5123
    0.8800    0.0095
    0.8900   -0.5289
    0.9000   -1.0350
    0.9100   -1.4456
    0.9200   -1.7142
    0.9300   -1.8198
    0.9400   -1.7680
    0.9500   -1.5864
    0.9600   -1.3154
    0.9700   -0.9962
    0.9800   -0.6611
    0.9900   -0.3279
    1.0000    0.0000
 /\relax}\relax
\setplotsymbol ({\ninerm .})
\setdots <3pt> 
\Red{\relax 
\plot
         0         0
    0.0100    0.3250
    0.0200    0.6460
    0.0300    0.9588
    0.0400    1.2597
    0.0500    1.5449
    0.0600    1.8110
    0.0700    2.0550
    0.0800    2.2742
    0.0900    2.4662
    0.1000    2.6293
    0.1100    2.7621
    0.1200    2.8638
    0.1300    2.9339
    0.1400    2.9728
    0.1500    2.9809
    0.1600    2.9595
    0.1700    2.9100
    0.1800    2.8345
    0.1900    2.7353
    0.2000    2.6149
    0.2100    2.4763
    0.2200    2.3226
    0.2300    2.1568
    0.2400    1.9823
    0.2500    1.8023
    0.2600    1.6201
    0.2700    1.4387
    0.2800    1.2611
    0.2900    1.0900
    0.3000    0.9277
    0.3100    0.7766
    0.3200    0.6383
    0.3300    0.5143
    0.3400    0.4057
    0.3500    0.3133
    0.3600    0.2373
    0.3700    0.1778
    0.3800    0.1344
    0.3900    0.1064
    0.4000    0.0929
    0.4100    0.0926
    0.4200    0.1040
    0.4300    0.1256
    0.4400    0.1556
    0.4500    0.1921
    0.4600    0.2333
    0.4700    0.2774
    0.4800    0.3224
    0.4900    0.3668
    0.5000    0.4088
    0.5100    0.4471
    0.5200    0.4804
    0.5300    0.5076
    0.5400    0.5279
    0.5500    0.5408
    0.5600    0.5458
    0.5700    0.5429
    0.5800    0.5322
    0.5900    0.5139
    0.6000    0.4885
    0.6100    0.4568
    0.6200    0.4195
    0.6300    0.3777
    0.6400    0.3322
    0.6500    0.2842
    0.6600    0.2347
    0.6700    0.1849
    0.6800    0.1359
    0.6900    0.0885
    0.7000    0.0437
    0.7100    0.0024
    0.7200   -0.0349
    0.7300   -0.0676
    0.7400   -0.0952
    0.7500   -0.1176
    0.7600   -0.1347
    0.7700   -0.1466
    0.7800   -0.1533
    0.7900   -0.1553
    0.8000   -0.1530
    0.8100   -0.1469
    0.8200   -0.1375
    0.8300   -0.1255
    0.8400   -0.1116
    0.8500   -0.0965
    0.8600   -0.0808
    0.8700   -0.0651
    0.8800   -0.0499
    0.8900   -0.0359
    0.9000   -0.0233
    0.9100   -0.0126
    0.9200   -0.0038
    0.9300    0.0028
    0.9400    0.0072
    0.9500    0.0096
    0.9600    0.0102
    0.9700    0.0091
    0.9800    0.0068
    0.9900    0.0036
    1.0000    0.0000
 /\relax}\relax
 \setdashes <4pt> 
\Blue{\relax 
\plot
         0         0
    0.0100    0.9592
    0.0200    1.8662
    0.0300    2.6725
    0.0400    3.3368
    0.0500    3.8279
    0.0600    4.1272
    0.0700    4.2294
    0.0800    4.1433
    0.0900    3.8903
    0.1000    3.5033
    0.1100    3.0230
    0.1200    2.4954
    0.1300    1.9674
    0.1400    1.4838
    0.1500    1.0836
    0.1600    0.7974
    0.1700    0.6452
    0.1800    0.6355
    0.1900    0.7645
    0.2000    1.0170
    0.2100    1.3673
    0.2200    1.7813
    0.2300    2.2194
    0.2400    2.6394
    0.2500    2.9999
    0.2600    3.2641
    0.2700    3.4021
    0.2800    3.3943
    0.2900    3.2327
    0.3000    2.9214
    0.3100    2.4764
    0.3200    1.9243
    0.3300    1.2995
    0.3400    0.6421
    0.3500   -0.0059
    0.3600   -0.6033
    0.3700   -1.1129
    0.3800   -1.5040
    0.3900   -1.7542
    0.4000   -1.8512
    0.4100   -1.7933
    0.4200   -1.5891
    0.4300   -1.2574
    0.4400   -0.8253
    0.4500   -0.3261
    0.4600    0.2027
    0.4700    0.7230
    0.4800    1.1983
    0.4900    1.5970
    0.5000    1.8940
    0.5100    2.0727
    0.5200    2.1258
    0.5300    2.0551
    0.5400    1.8714
    0.5500    1.5928
    0.5600    1.2432
    0.5700    0.8506
    0.5800    0.4441
    0.5900    0.0528
    0.6000   -0.2970
    0.6100   -0.5836
    0.6200   -0.7906
    0.6300   -0.9088
    0.6400   -0.9358
    0.6500   -0.8766
    0.6600   -0.7422
    0.6700   -0.5488
    0.6800   -0.3163
    0.6900   -0.0662
    0.7000    0.1796
    0.7100    0.4014
    0.7200    0.5821
    0.7300    0.7092
    0.7400    0.7749
    0.7500    0.7769
    0.7600    0.7181
    0.7700    0.6059
    0.7800    0.4520
    0.7900    0.2704
    0.8000    0.0771
    0.8100   -0.1120
    0.8200   -0.2823
    0.8300   -0.4214
    0.8400   -0.5205
    0.8500   -0.5744
    0.8600   -0.5822
    0.8700   -0.5469
    0.8800   -0.4751
    0.8900   -0.3758
    0.9000   -0.2599
    0.9100   -0.1391
    0.9200   -0.0247
    0.9300    0.0731
    0.9400    0.1466
    0.9500    0.1902
    0.9600    0.2021
    0.9700    0.1833
    0.9800    0.1383
    0.9900    0.0742
    1.0000    0.0000
 /\relax}\relax
\endpicture
}
\xfiglen=4.5 true in
\yfiglen=2.25true in
\setbox\figurefour=\vbox{\hsize=\xfiglen
\beginpicture
\footnotesize
  \setcoordinatesystem units <\xfiglen,0.15\yfiglen> 
  \setplotarea x from 0 to 1, y from -2 to 5.4
  \axis bottom shiftedto y=0 ticks short numbered from 0 to 1 by 0.2 /
  \axis left ticks short numbered from -2 to 5 by 1 /
\small
\put {\copy\figurereconlegendthree} [rt] at 0.98  5
\put {$x$} [lb] at 1 0.06
\put {$u_0(x)$} [lb] at 0.01 5.3
\setquadratic
\setsolid
\setplotsymbol ({\eightrm .})
\Black{\relax 
\plot
         0         0
    0.0100    1.4947
    0.0200    2.8232
    0.0300    3.8561
    0.0400    4.5263
    0.0500    4.8384
    0.0600    4.8577
    0.0700    4.6850
    0.0800    4.4250
    0.0900    4.1578
    0.1000    3.9223
    0.1100    3.7161
    0.1200    3.5083
    0.1300    3.2615
    0.1400    2.9528
    0.1500    2.5883
    0.1600    2.2047
    0.1700    1.8587
    0.1800    1.6093
    0.1900    1.4988
    0.2000    1.5388
    0.2100    1.7071
    0.2200    1.9542
    0.2300    2.2183
    0.2400    2.4418
    0.2500    2.5849
    0.2600    2.6310
    0.2700    2.5843
    0.2800    2.4610
    0.2900    2.2781
    0.3000    2.0451
    0.3100    1.7614
    0.3200    1.4198
    0.3300    1.0147
    0.3400    0.5512
    0.3500    0.0508
    0.3600   -0.4487
    0.3700   -0.8992
    0.3800   -1.2526
    0.3900   -1.4718
    0.4000   -1.5388
    0.4100   -1.4569
    0.4200   -1.2470
    0.4300   -0.9404
    0.4400   -0.5695
    0.4500   -0.1611
    0.4600    0.2663
    0.4700    0.6999
    0.4800    1.1283
    0.4900    1.5359
    0.5000    1.9000
    0.5100    2.1918
    0.5200    2.3819
    0.5300    2.4473
    0.5400    2.3784
    0.5500    2.1828
    0.5600    1.8842
    0.5700    1.5175
    0.5800    1.1211
    0.5900    0.7283
    0.6000    0.3633
    0.6100    0.0387
    0.6200   -0.2409
    0.6300   -0.4745
    0.6400   -0.6601
    0.6500   -0.7916
    0.6600   -0.8599
    0.6700   -0.8566
    0.6800   -0.7788
    0.6900   -0.6333
    0.7000   -0.4382
    0.7100   -0.2198
    0.7200   -0.0076
    0.7300    0.1727
    0.7400    0.3050
    0.7500    0.3849
    0.7600    0.4195
    0.7700    0.4224
    0.7800    0.4081
    0.7900    0.3871
    0.8000    0.3633
    0.8100    0.3350
    0.8200    0.2983
    0.8300    0.2514
    0.8400    0.1979
    0.8500    0.1475
    0.8600    0.1135
    0.8700    0.1083
    0.8800    0.1386
    0.8900    0.2014
    0.9000    0.2846
    0.9100    0.3691
    0.9200    0.4350
    0.9300    0.4668
    0.9400    0.4578
    0.9500    0.4113
    0.9600    0.3379
    0.9700    0.2513
    0.9800    0.1631
    0.9900    0.0794
    1.0000    0.0000
 /\relax}\relax
\setplotsymbol ({\ninerm .})
\setdots <3pt> 
\Red{\relax 
\plot
         0         0
    0.0100    0.4719
    0.0200    0.9368
    0.0300    1.3879
    0.0400    1.8186
    0.0500    2.2226
    0.0600    2.5944
    0.0700    2.9289
    0.0800    3.2216
    0.0900    3.4692
    0.1000    3.6688
    0.1100    3.8187
    0.1200    3.9180
    0.1300    3.9669
    0.1400    3.9661
    0.1500    3.9177
    0.1600    3.8242
    0.1700    3.6891
    0.1800    3.5166
    0.1900    3.3113
    0.2000    3.0785
    0.2100    2.8235
    0.2200    2.5523
    0.2300    2.2707
    0.2400    1.9845
    0.2500    1.6995
    0.2600    1.4212
    0.2700    1.1547
    0.2800    0.9047
    0.2900    0.6752
    0.3000    0.4698
    0.3100    0.2912
    0.3200    0.1416
    0.3300    0.0224
    0.3400   -0.0660
    0.3500   -0.1235
    0.3600   -0.1512
    0.3700   -0.1504
    0.3800   -0.1233
    0.3900   -0.0726
    0.4000   -0.0014
    0.4100    0.0867
    0.4200    0.1881
    0.4300    0.2987
    0.4400    0.4146
    0.4500    0.5317
    0.4600    0.6462
    0.4700    0.7545
    0.4800    0.8533
    0.4900    0.9395
    0.5000    1.0107
    0.5100    1.0650
    0.5200    1.1008
    0.5300    1.1172
    0.5400    1.1139
    0.5500    1.0912
    0.5600    1.0497
    0.5700    0.9907
    0.5800    0.9160
    0.5900    0.8277
    0.6000    0.7283
    0.6100    0.6204
    0.6200    0.5070
    0.6300    0.3911
    0.6400    0.2756
    0.6500    0.1637
    0.6600    0.0580
    0.6700   -0.0388
    0.6800   -0.1244
    0.6900   -0.1970
    0.7000   -0.2548
    0.7100   -0.2969
    0.7200   -0.3224
    0.7300   -0.3313
    0.7400   -0.3237
    0.7500   -0.3004
    0.7600   -0.2626
    0.7700   -0.2117
    0.7800   -0.1497
    0.7900   -0.0787
    0.8000   -0.0011
    0.8100    0.0805
    0.8200    0.1634
    0.8300    0.2451
    0.8400    0.3229
    0.8500    0.3944
    0.8600    0.4573
    0.8700    0.5096
    0.8800    0.5496
    0.8900    0.5761
    0.9000    0.5880
    0.9100    0.5849
    0.9200    0.5667
    0.9300    0.5338
    0.9400    0.4868
    0.9500    0.4271
    0.9600    0.3561
    0.9700    0.2757
    0.9800    0.1880
    0.9900    0.0953
    1.0000    0.0000
 /\relax}\relax
\setdashes <3pt>
\OliveGreen{\relax 
\plot
         0         0
    0.0100    0.7115
    0.0200    1.3952
    0.0300    2.0226
    0.0400    2.5655
    0.0500    3.0001
    0.0600    3.3132
    0.0700    3.5071
    0.0800    3.6026
    0.0900    3.6351
    0.1000    3.6460
    0.1100    3.6701
    0.1200    3.7231
    0.1300    3.7952
    0.1400    3.8514
    0.1500    3.8428
    0.1600    3.7229
    0.1700    3.4666
    0.1800    3.0832
    0.1900    2.6202
    0.2000    2.1539
    0.2100    1.7709
    0.2200    1.5426
    0.2300    1.5041
    0.2400    1.6413
    0.2500    1.8923
    0.2600    2.1630
    0.2700    2.3518
    0.2800    2.3767
    0.2900    2.1968
    0.3000    1.8215
    0.3100    1.3050
    0.3200    0.7295
    0.3300    0.1813
    0.3400   -0.2714
    0.3500   -0.5930
    0.3600   -0.7829
    0.3700   -0.8672
    0.3800   -0.8825
    0.3900   -0.8597
    0.4000   -0.8119
    0.4100   -0.7314
    0.4200   -0.5960
    0.4300   -0.3819
    0.4400   -0.0775
    0.4500    0.3069
    0.4600    0.7371
    0.4700    1.1618
    0.4800    1.5247
    0.4900    1.7799
    0.5000    1.9033
    0.5100    1.8973
    0.5200    1.7876
    0.5300    1.6141
    0.5400    1.4180
    0.5500    1.2305
    0.5600    1.0667
    0.5700    0.9249
    0.5800    0.7924
    0.5900    0.6534
    0.6000    0.4974
    0.6100    0.3238
    0.6200    0.1426
    0.6300   -0.0293
    0.6400   -0.1738
    0.6500   -0.2773
    0.6600   -0.3344
    0.6700   -0.3480
    0.6800   -0.3272
    0.6900   -0.2837
    0.7000   -0.2278
    0.7100   -0.1671
    0.7200   -0.1062
    0.7300   -0.0479
    0.7400    0.0045
    0.7500    0.0470
    0.7600    0.0753
    0.7700    0.0868
    0.7800    0.0830
    0.7900    0.0701
    0.8000    0.0590
    0.8100    0.0619
    0.8200    0.0888
    0.8300    0.1435
    0.8400    0.2212
    0.8500    0.3102
    0.8600    0.3942
    0.8700    0.4582
    0.8800    0.4929
    0.8900    0.4976
    0.9000    0.4798
    0.9100    0.4522
    0.9200    0.4271
    0.9300    0.4116
    0.9400    0.4045
    0.9500    0.3967
    0.9600    0.3746
    0.9700    0.3255
    0.9800    0.2432
    0.9900    0.1306
    1.0000   -0.0000
 /\relax}\relax
\setdashpattern <3pt,2pt,1pt,2pt> 
\Blue{\relax 
\plot
         0         0
    0.0100    0.9375
    0.0200    1.8302
    0.0300    2.6339
    0.0400    3.3070
    0.0500    3.8162
    0.0600    4.1426
    0.0700    4.2877
    0.0800    4.2761
    0.0900    4.1514
    0.1000    3.9671
    0.1100    3.7726
    0.1200    3.6004
    0.1300    3.4572
    0.1400    3.3244
    0.1500    3.1671
    0.1600    2.9496
    0.1700    2.6538
    0.1800    2.2916
    0.1900    1.9084
    0.2000    1.5742
    0.2100    1.3652
    0.2200    1.3396
    0.2300    1.5170
    0.2400    1.8672
    0.2500    2.3125
    0.2600    2.7444
    0.2700    3.0498
    0.2800    3.1386
    0.2900    2.9658
    0.3000    2.5409
    0.3100    1.9227
    0.3200    1.2016
    0.3300    0.4752
    0.3400   -0.1747
    0.3500   -0.6974
    0.3600   -1.0775
    0.3700   -1.3272
    0.3800   -1.4716
    0.3900   -1.5326
    0.4000   -1.5180
    0.4100   -1.4186
    0.4200   -1.2149
    0.4300   -0.8889
    0.4400   -0.4384
    0.4500    0.1151
    0.4600    0.7242
    0.4700    1.3243
    0.4800    1.8465
    0.4900    2.2335
    0.5000    2.4523
    0.5100    2.4993
    0.5200    2.3977
    0.5300    2.1880
    0.5400    1.9153
    0.5500    1.6179
    0.5600    1.3198
    0.5700    1.0302
    0.5800    0.7474
    0.5900    0.4670
    0.6000    0.1884
    0.6100   -0.0807
    0.6200   -0.3244
    0.6300   -0.5223
    0.6400   -0.6560
    0.6500   -0.7142
    0.6600   -0.6963
    0.6700   -0.6122
    0.6800   -0.4793
    0.6900   -0.3182
    0.7000   -0.1485
    0.7100    0.0141
    0.7200    0.1579
    0.7300    0.2746
    0.7400    0.3580
    0.7500    0.4032
    0.7600    0.4073
    0.7700    0.3713
    0.7800    0.3015
    0.7900    0.2106
    0.8000    0.1164
    0.8100    0.0383
    0.8200   -0.0076
    0.8300   -0.0122
    0.8400    0.0241
    0.8500    0.0914
    0.8600    0.1741
    0.8700    0.2557
    0.8800    0.3240
    0.8900    0.3741
    0.9000    0.4086
    0.9100    0.4346
    0.9200    0.4594
    0.9300    0.4852
    0.9400    0.5075
    0.9500    0.5147
    0.9600    0.4924
    0.9700    0.4287
    0.9800    0.3195
    0.9900    0.1711
    1.0000   -0.0000
 /\relax}\relax
\endpicture
}

\section{Regularization strategies}\label{sec:regularization_strategies}

We have outlined several possible candidates for a {\it quasi-reversible\/}
regularizer for the backwards heat equation. 
In this section we provide an overall strategy and look at how individual
regularizing equations fit in.

The ultimate idea is to split the 
\Margin{Report C (18)}
problem into distinct frequency bands
and then combine to recover the value of $u_0(x)$.
This is feasible since the mapping $u_0\to g(x)$ is linear.
Such a strategy is of course not new for this problem but
the key is to recognize that each {\it quasi-reversible\/} component
that we have described will perform differently over each frequency band
and the problem is how to make the most effective combination.

Throughout we assume that the final value $g(x)$ has been measured subject
to a noise level, the magnitude of which we know.
Clearly, knowing further information such as some of the moments of the 
probability density function of the noise is desirable, but we will
simply assume that it has mean zero and a known maximum expected value, 
which leads to the deterministic noise bound
\begin{equation}\label{eqn:delta}
\|\uT-\uTdel\|_{L^2(\Omega)}\leq \delta
\end{equation}
with given $\delta>0$.

\subsection{Using a subdiffusion regularization}\label{subdiffusion_reg}

Perhaps the simplest possibility of regularization by a subdiffusion process
is to replace the time derivative
in \eqref{eqn:direct_prob} by one of fractional order $\partial_t^\alpha$, 
relying on the stability estimate from Lemma~\ref{lem:mlf-stability_est}.
However, there are two obstacles to this.

First, from \eqref{lem:subdiff_smoothness} there is still some smoothing
of the subdiffusion operator and  the actual final value at $t=T$
will lie in 
\revision{$\dH{2}$}.
\Margin{Report C (15)}
The subdiffusion equation still decays to zero for large $T$
and indeed the amplification factor $A_{frac}(k,\alpha)$
connecting the Fourier coefficients
\begin{equation}\label{eqn:subdiff-reg}
A_{frac}(k,\alpha) \langle \uT,\phi_k\rangle = \langle u(\cdot,0;\alpha),\phi_k\rangle
\end{equation}
is
\begin{equation}\label{eqn:amplification1}
A_{frac}(k,\alpha)=1/\revision{E_{\alpha,1}}(-\lambda_k T^\alpha) 
\end{equation}
and thus grows linearly in $\lambda_k$.

Thus we must form $\uTtd(x)$ the projection of the data $\uTdel(x)$
onto 
\revision{$\dH{2}$},
\Margin{Report C (15)}
in order that the amplification remain bounded for all 
$\lambda_k$. 
This is easily accomplished by some conventional regularization method.
Tikhonov regularization (or iterated versions of it) is not appropriate for this purpose, since due to its saturation at a finite smoothness level, it would not be able to optimally exploit the fact that we deal with infinitely smooth exact data $\uT$.
Thus we employ Landweber iteration for this purpose
\begin{equation}\label{eqn:LW_reg}
w^{(i+1)}=w^{(i)}-\mu (-\mathbb{L})^{-2}(w^{(i)}-\uTdel)\,, \qquad w^{(0)}=0\,,
\end{equation}
which by setting $w^{(i)}=-\mathbb{L} p^{(i)}$ can be interpreted as a gradient descent method for the minimization problem 
\[
\min_{p\in L^2(\Omega)}  \frac12 \|(-\mathbb{L})^{-1}p-\uTdel\|_{L^2(\Omega)}^2\,,
\]
and set 
\[
\uTtd=w^{(i_*)}
\]
for some appropriately chosen index $i_*$.
In practice we use the discrepancy principle for this purpose, while the convergence result in Lemma~\ref{lem:smoothing} employs an a priori choice of $i_*$. 
In \eqref{eqn:LW_reg}, the step size $\mu$ is assumed to satisfy
\[
\mu\in \left(0,\frac{1}{\|(-\mathbb{L})^{-2}\|_{L^2\to L^2}}\right]\,.
\]

Second, one has to check that none of the amplification coefficients
in \eqref{eqn:amplification1} exceeds that for the heat equation itself.
However, as shown in \cite{JinRundell:2015} for any value of $T$
there exists an $N$ such 
\revision{that} 
\Margin{Report C (19)}
all amplification factors $A_{frac}(k,\alpha)$
in \eqref{eqn:amplification1} exceed those of the 
parabolic problem
\[
A_{par}(k) := e^{\lambda_k T}
\]
for $k\leq N$.
In this sense the low frequencies are more difficult to recover
by means of the regularizing subdiffusion equation than by the parabolic
equation itself.
Of course for large values of $\lambda$ the situation reverses
as the Mittag-Leffler function decays only linearly for large argument.
Figure~\ref{fig:amplification_factors} shows the plots of $A_{frac}(k,\alpha)$
for $\alpha=0.5,\;0.9,\;1$.

\begin{figure}[ht]
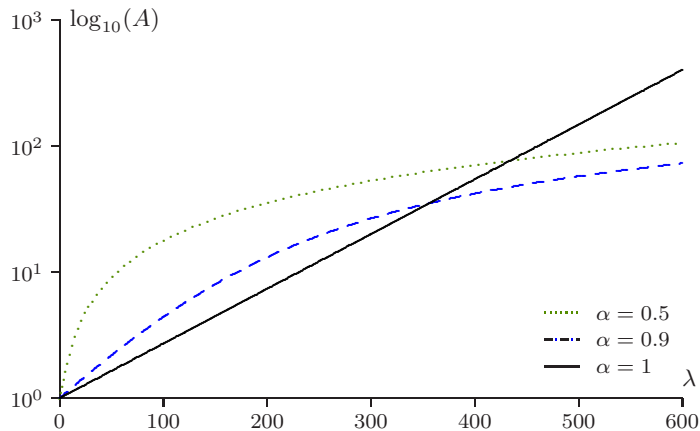
 
\hbox to \hsize{\hss\box\figureone\hss}
\caption{\small Amplification factor $A(\lambda_k,\alpha)$
\label{fig:amplification_factors}}
\end{figure}

\medskip
Thus we must modify the reconstruction scheme and there are several
possibilities of which we describe two.

\subsection{Adding a fractional time derivative to the diffusion equation}

The first of these is to take a multi-term fractional derivative
replacing \eqref{eqn:subdiffusion} by
\begin{equation}\label{eqn:epsilon-multi-term}
u_t + \epsilon\partial_t^\alpha u - \mathbb{L} u = 0,
\quad (x,t)\in\Omega\times(0,T).
\end{equation}
We must show that the solution $u_\epsilon$ to \eqref{eqn:epsilon-multi-term} 
and subject to the same initial condition converges in $L^2(\Omega)\times(0,T)$
to the solution of \eqref{eqn:direct_prob} as $\epsilon\to 0$.
Equation \eqref{eqn:epsilon-multi-term} is a specific case of the more general
multiterm fractional diffusion operator
\begin{equation}\label{eqn:multi-term}
\sum_{j=1}^M q_j\partial_t^{\alpha_j} u - \mathbb{L} u = 0,
\quad (x,t)\in\Omega\times(0,T).
\end{equation}
In this case the Mittag-Leffler function must be replaced by the multiterm
version, \cite{LuchkoGorenflo:1999,LiLiuYamamoto:2015}
with considerable additional complications
although the theory is now well-understood.
While one can use \eqref{eqn:multi-term} the complexity here arises from
the $2M$ coefficients $\{q_j,\alpha_j\}$ that would have to be determined
as part of the regularization process.
Thus we restrict our attention to equation \eqref{eqn:epsilon-multi-term}.

We can calculate the fundamental solution to 
\eqref{eqn:epsilon-multi-term} as follows.
First we consider the relaxation equation
\begin{equation}\label{eqn:two-term-relaxation}
w' + \epsilon\partial_t^\alpha w  + \lambda w  = 0,
\qquad w(0) = 1.
\end{equation}
Taking Laplace transforms $\{t\rightarrow s\}$  we obtain 
\begin{equation}\label{eqn:two-term-relaxation_lap}
\hat w(s) = \frac{1 + \epsilon s^{\alpha-1}}{s + \epsilon s^\alpha + \lambda}
\end{equation}
Now the imaginary part of $s + 
\revision{\epsilon} 
s^\alpha + \lambda$  does not vanish if $s$ is
\Margin{Report C (20)}
not real and positive so that the inversion of the Laplace transform can be
accomplished by deforming the original vertical Bromwich path into a
Hankel path $\mathcal{H}_\eta$
surrounding the branch cut on the negative real axis and a small circle
of radius $\eta$ centre the origin.
See, Chapter~4 of \cite{GorenfloMainardi:1997}.
This gives 
\begin{equation*}
w(t) = \frac{1}{2\pi i}\int_{\mathcal{H}_\eta}
e^{st} \frac{1 + \epsilon s^{\alpha-1}}{s + \epsilon s^\alpha + \lambda}\,ds
\end{equation*}
and as $\eta\to 0$ we obtain $H(r) = -\frac{1}{\pi}
Im\Bigl\{ \frac{1 + \epsilon s^{\alpha-1}}{s + \epsilon s^\alpha + \lambda}\Bigr\}\Big|_{s=re^{i\pi}}$.
Multiplying both numerator and denominator by the complex conjugate of
$s + \epsilon s^\alpha +\lambda$ gives
\begin{equation*}
\begin{aligned}
\Bigl\{ \frac{1 + \epsilon s^{\alpha-1}}{s + \epsilon s^\alpha + \lambda}\Bigr\}\Big|_{s=re^{i\pi}} &= 
\frac{1 + \epsilon r^{\alpha-1}e^{(\alpha-1)\pi i}}{(\lambda-r) + \epsilon r^\alpha e^{\alpha\pi i}} \\
&= \frac{\bigl(1 + \epsilon r^{\alpha-1}e^{(\alpha-1)\pi i}\bigr)
\bigl( (\lambda-r) + \epsilon r^\alpha e^{-\alpha\pi i}\bigr)}
{(\lambda-r)^2 + 2 \epsilon(\lambda-r) r^\alpha\cos(\alpha\pi) + \epsilon^2r^{2\alpha}} \\
&= \frac{
(\lambda-r) -\epsilon^2 r^{2\alpha-1} + 2 \epsilon r^\alpha\cos(\alpha\pi)
-\epsilon\lambda r^{\alpha-1}e^{\alpha\pi i} }
{(\lambda-r)^2 + 2 \epsilon(\lambda-r) r^\alpha\cos(\alpha\pi) + \epsilon^2r^{2\alpha}}\,. \\
\end{aligned}
\end{equation*}
The first three terms in the numerator are real so
taking the imaginary part yields
\begin{equation*}
Im\Bigl\{ \frac{1 + \epsilon s^{\alpha-1}}{s + \epsilon s^\alpha + \lambda}\Bigr\}\Big|_{s=re^{i\pi}}
= 
-\frac{\epsilon\lambda r^{\alpha-1}\sin(\alpha\pi)}
{(\lambda-r)^2 + 2\epsilon(\lambda-r) r^\alpha\cos(\alpha\pi) + \epsilon^2r^{2\alpha}}\,.
\end{equation*}
Thus
\begin{equation}\label{eqn:two-term-relaxation_sol}
w(t;\alpha,\epsilon,\lambda) = \int_0^\infty e^{-rt} H(r;\alpha,\epsilon,\lambda)\,dr
\end{equation}
where the spectral function $H$ satisfies
\begin{equation}\label{eqn:two-term-spec-func}
H(r;\alpha,\epsilon,\lambda) = \frac{\epsilon\lambda}{\pi}\frac{r^{\alpha-1}\sin(\alpha\pi)}
{(\lambda-r)^2 + \epsilon^2 r^{2\alpha} + 2(\lambda-r)\epsilon r^\alpha\cos(\alpha\pi)}\,.
\end{equation}
For $\epsilon>0$ and $0<\alpha<1$, $H$ is strictly positive showing that the
fundamental solution of the initial value problem
\eqref{eqn:two-term-relaxation} is also a completely monotone function.

We can use the 
\Margin{Report C (21)}
above to obtain the solution representation to equation
\eqref{eqn:multi-term}
\begin{equation}\label{eqn:two-term-solution}
u(x,t) = \sum_{k=1}^\infty \langle u_0,\phi_k\rangle
w(t,\alpha,\epsilon,\lambda_k) \phi_k(x)
\end{equation}
which becomes a potential regularizer for the backwards heat problem.

Tauberian results for the Laplace transform $\hat f(s)$ of a sufficiently
smooth function $f(\tau)$ show that
$\lim_{\tau\to\infty} f(\tau) =\lim_{s\to 0^+} \hat f(s)$.
If in \eqref{eqn:two-term-spec-func} we make the change of variables 
$r\to \lambda \rho$, $H(r,\cdot) \to \tilde H(\rho,\cdot)$
then consider $\lim_{\rho\to 0^+}\tilde H(\rho)$.
We obtain
$\tilde H(\rho) \sim{}
\frac{\epsilon\sin(\alpha\pi)}{\pi}\lambda^{\alpha-1}\rho^{\alpha-1}$.
Now hold $t=T$ fixed in \eqref{eqn:two-term-relaxation_sol} and we see that
\begin{equation}\label{eqn:w-asymptotic}
\begin{aligned}
w(T;\alpha,\epsilon,\lambda) &\sim \epsilon \frac{\sin(\alpha\pi)}{\pi}
\lambda^{\alpha-1}
\frac{(\lambda T)^{-\alpha}}{\Gamma(1-\alpha)}\\
&\sim C(\alpha)\epsilon \frac{T^{-\alpha}}{\lambda}\qquad 
\mbox{as} \;\ \lambda\to\infty\,.\\
\end{aligned}
\end{equation}
This indicates that the combined asymptotic behaviour of the two fractional
terms in \eqref{eqn:epsilon-multi-term}
defers to that of the lower fractional index, here $\alpha$.
This is in fact known even for the general multiterm case
\eqref{eqn:multi-term}, see \cite{LiLiuYamamoto:2015} .

Thus given we have made the prior regularization of the data by mapping it
into 
\revision{$\dH{2}$},
\Margin{Report C (15)}
equation \eqref{eqn:epsilon-multi-term} will be a
regularization method for the diffusion equation for $\epsilon>0$.
The question then becomes how effective it performs.

The answer is, quite poorly and \eqref{eqn:w-asymptotic} shows why.
If $\epsilon$ is very small then the asymptotic decay of the singular
values of the map $F:u_0\to g$ is again too great and the combination
of the two derivatives is insufficent to control the high frequencies.
This is particularly true the closer $\alpha$ is to unity.
On the other hand, for lower frequency values of $\lambda$, the fractional
derivative term plays a considerable role and the greater with increasing
$\epsilon$ and decreasing $\alpha$.
Thus one is forced to select regularizing constants $\alpha$ and $\epsilon$
that will either decrease fidelity at the lower frequencies or fail to
adequately control the high frequencies.

There is a partial solution to the above situation by taking instead of
\eqref{eqn:epsilon-multi-term} the balanced version
\begin{equation}\label{eqn:epsilon-multi-term-balanced}
(1-\epsilon)u_t + \epsilon\partial_t^\alpha u - \mathbb{L} u = 0,
\quad (x,t)\in\Omega\times(0,T).
\end{equation}
This ameliorates to some degree the concern at lower frequencies
but has little effect at the higher frequencies.

We will not dwell on this version or its above modification
as there are superior alternatives as will see in the next subsection.
However, the lessons learned in the previous two versions shows the way
to achieve both goals; low frequency fidelity and high frequency control.

\subsection{Using split-frequencies}\label{sec:split_frequency_case}

Another  alternative is to modify the reconstruction scheme as follows:
for frequencies $k\leq K$ we recover the Fourier coefficients of $u_0$
by simply inverting the parabolic equation 
\revision{as} 
\Margin{Report C (22)}
is, using
$A_{par}(k)$ and for frequencies $k>K$ we use $A_{frac}(k,\alpha)$ defined in
\eqref{eqn:amplification1}, i.e.,
\begin{equation*}
\langle u_0^\delta(\cdot;\alpha,K),\phi_k\rangle=\begin{cases}
A_{par}(k) \langle \uTdel,\phi_k\rangle &\mbox{ for }k\leq K\\
A_{frac}(k,\alpha) \langle \uTtd,\phi_k\rangle &\mbox{ for }k\geq K+1
\end{cases}
\end{equation*}

The question remains how to pick $K$ and $\alpha$.
For this purpose, we use the discrepancy principle: in both cases using the 
assumption on the noise level in $\uTdel$ and its smoothed version $\uTtd$, respectively.
More precisely, we first of all apply the discrepancy principle to find $K$, 
which -- according to existing results on truncated singular value expansion,
see for example, \cite{EnglHankeNeubauer:1996} -- gives an order-optimal
(with respect to the $L^2$ norm) low frequency reconstruction
$u_{0,lf}^\delta(\cdot,K)$.
Then we aim at improving this reconstruction by adding higher frequency
components that cannot be recovered by the pure backwards heat equation,
which is enabled by a subdiffusion regularization acting only on
these frequencies.
The exponent $\alpha$ acts as a regularization parameter that is again chosen
by the discrepancy principle.
We refer to Section~\ref{sec:conv_split} for details on this procedure.

This works remarkably well for a wide range of functions $u_0$.
It works less well if the initial value contains a significant amount of
mid-level frequencies as well as those of low and high order.
In this case the {\it split-frequency\/} idea can be adapted as follows.

As above we determine the value of $K_1=K$ using the discrepancy principle;
this is the largest frequency mode that can be inverted using the parabolic
amplification $A_{par}(k)$ given the noise level $\delta$.
We then  estimate $K_2$ which will be the boundary between the mid and high
frequencies. With this estimate we again use the discrepancy principle to
determine the optimal $\alpha_2$ where we will use $A_{frac}(k,\alpha_2)$
for those frequencies above $K_2$ to recover the Fourier coefficients of $u_0$
for $k>K_2$.
In practice we set a maximum frequency value $K_\ell$.
By taking various values of $K_2$ we perform the above to obtain the
overall best fit in the above scheme.

To regularize the mid frequency range we again use the subdiffusion equation
with $\alpha=\alpha_1$ and choose this parameter by again using
the discrepancy principle.
I.e., we set 
\begin{equation*}
\langle u_0^\delta(\cdot;\alpha_1,\alpha_2,K_1,K_2),\phi_k\rangle=\begin{cases}
A_{par}(k) \langle \uTdel,\phi_k\rangle &\mbox{ for }k\leq K_1\\
A_{frac}(k,\alpha_1) \langle \uTtd,\phi_k\rangle &\mbox{ for }K_1+1\leq k\leq K_2\\
A_{frac}(k,\alpha_2) \langle \uTtd,\phi_k\rangle &\mbox{ for }K_2+1\leq k
\revision{\leq} 
K_\ell\\
0 &\mbox{ for }k\geq K_\ell+1
\end{cases}
\end{equation*}
\Margin{Report C (23)}

Thus we solve the backwards diffusion equation in three frequency ranges
$(1,K_1)$, $(K_1,K_2)$ and $(K_2,K_\ell)$ using 
\eqref{eqn:amplification1} with $\alpha=1$, $\alpha=\alpha_1$ and $\alpha_2$.

We remark that this process could be extended whereby we split the frequencies
into $[1,K_1]$,
$(K_1,K_2]$, $\ldots$ $(
\revision{K_{\ell-1}} 
,K_\ell)$ and use the discrepancy principle
\Margin{Report C (24)}
to obtain a sequence of values $\alpha=1,\; \alpha_1,\;\dots\; \alpha_m$.
We found that in general the values of $\alpha_i$ decreased
with increasing frequency.  This is to be expected; although the asymptotic
order of $E_{\alpha,1}$ is the same for all $\alpha<1$ the
associated constant is not; larger values of $\alpha$ correspond a larger
constant and give higher fidelity with the heat equation as Lemmas~\ref{lem:mlf-stability_est} and \ref{lem:rateE} show.

\subsection{Using space fractional regularization}\label{sec:space_frac_reg}

The idea of the previous subsection can be carried over to fractional operators
in space. Once again we look for frequency cut-off values and we illustrate
with 3 levels, so we have $K_1$ and $K_2$ as above.
The regularizing equation will be the $\beta-$pseudoparabolic as in
\eqref{eqn:beta-pseudoparabolic}.
For the lowest frequency interval we choose $\epsilon=0$ so that we
are simply again inverting the parabolic.
for the mid range $k\in (K_1,K_2]$ we take $\beta=0.5$ and for the high
frequencies $k\in (K_2,K_\ell]$ we use $\beta=1$ so that we have the usual
pseudoparabolic equation.
\begin{equation*}
\langle u_0^\delta(\cdot;\beta_1,\beta_2,\epsilon_1,\epsilon_2,K_1,K_2),\phi_k\rangle=\begin{cases}
A_{par}(k) \langle \uTdel,\phi_k\rangle &\mbox{ for }k\leq K_1\\
A_{\beta ps}(k,\beta_1,\epsilon_1) \langle \uTdel,\phi_k\rangle &\mbox{ for }K_1+1\leq k\leq K_2\\
A_{\beta ps}(k,\beta_2,\epsilon_2) \langle \uTdel,\phi_k\rangle &\mbox{ for }K_2+1\leq k\geq K_\ell\\
0 &\mbox{ for }k\geq K_\ell+1
\end{cases}
\end{equation*}
with 
\[
A_{\beta ps}(k,\beta,\epsilon)=\exp\left(\frac{\lambda_n}{1 + \epsilon\lambda_n^\beta} T\right)
\]

This is a regularizer in $L^2(\Omega)$ and so there is no need
for the preliminary mapping of the data into 
\revision{$\dH{2}$}.
\Margin{Report C (15)}
In each interval we compute the value of $\epsilon$ from the discrepancy
principle and invert the corresponding amplification factors to recover
$u_0$ from \eqref{eqn:beta-pseudoparabolic_sol}.

Variations are possible and in particular reserving the $\beta-$pseudoparabolic
equation for mid-range frequencies and using a subdiffusion equation for the
regularization of the high frequencies as in
Section~\ref{sec:split_frequency_case}.

\section{Reconstructions}\label{sec:reconstructions}

In this section we will show a few illustrative examples 
for $\mathbb{L}=\triangle$ in one space dimension based on
inversion using the split-frequency model incorporating fractional diffusion
operators since overall these gave the best reconstructions of
the initial data.
Comparisons between different methods is always subject to the possibility
that, given almost any inversion method, one can construct an initial function
$u_0$ that will reconstruct well for that method.
As noted in the previous section we did find \eqref{eqn:epsilon-multi-term}
or its modification \eqref{eqn:epsilon-multi-term-balanced} to be competitive.
The $\beta$-pseudoparabolic equation \eqref{eqn:beta-pseudoparabolic}
when used only for mid range frequencies
and with a subdiffusion operator for the high frequencies can give
comparable results to the 
double-split fractional. 
However, the difficulty lies in determining the pair of constants
$\beta$ and $\epsilon$.  It turns out that the optimal reconstruction using
the discrepancy principle is not sensitive to $\beta$ in the range 
$\frac{1}{2}\leq\beta\leq\frac{3}{4}$ or even beyond, but it is sensitive
to the choice of $\epsilon$.

We have taken two noise levels $\delta$ on the data $\uT(x)=u(x,T)$
at which to show recovery of $u_0$;
$\delta = 1\%$ and $\delta = 0.1\%$.
These may seem a low noise level but one must understand the high degree
of ill-posedness of the problem and the fact that high Fourier modes
very quickly become damped beyond any reasonable measurement level.
One is reminded here of the quote by Lanczos,
{\it ``lack of information cannot be remedied by any mathematical trickery''}

We have also taken the final time $T$ to be $T=0.02$.
As noted in the introduction concerning equation scaling we in reality
have the combination $Kt$ for the parabolic equation in
\eqref{eqn:direct_prob} where $K$ is typically quite small.
Since we have set $K=1$ here our choice of final time is actually rather long.
A decrease in our $T$ by a factor of ten would result in much superior
reconstructions for the same level of data noise.

Our first example is of a smooth function except for a discontinuity in
its derivative near the rightmost endpoint so that recovery of relatively high
frequency information is required in order to resolve this feature.
Figure~\ref{fig:reconstructions1} shows the actual function $u_0$ together
with reconstructions from both the single split-frequency method
and with a double splitting. 
One sees the slight but significant resolution increase for the latter method.
In the single split method the discrepancy principle chose
$K_1=4$ and $\alpha=0.92$; for the double split 
$K_1=4$, $K_2=10$ and $\alpha_1=0.999$, $\alpha_2=0.92$.

It is worth noting that if we had to increase the noise to $\delta=0.01$
then not only would the reconstruction degrade but would do so more clearly
near the singularity in the derivative. However, of more interest is the  fact
that the reconstructions from both methods would be identical.
The discrepancy principle detects there is insufficient information for
a second splitting, so that $K_2$ is taken to be equal to $K_1$.

\begin{figure}[ht]
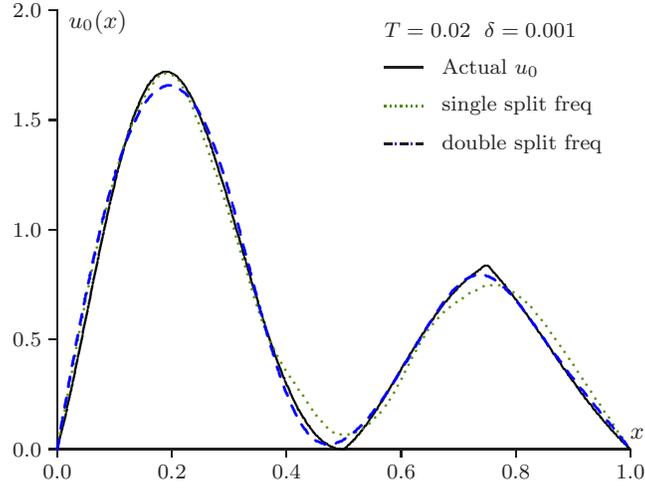

\hbox to \hsize{\hss\box\figuretwo\hss}
\caption{\small Reconstructions from single and double split frequency method.} 
\label{fig:reconstructions1}
\end{figure}

As a second example we chose a function made up by setting its Fourier
coefficients and choosing these so that the first 7 are all around unity
as are those in the range from 10 to 15.
The reconstructions shown in Figure~\ref{fig:reconstructions2} are using
the triple interval split frequency and as a comparison a truncated
singular value decomposition from the parabolic equation
with the parameters chosen again by the discrepancy principle.
As the figure shows, the {\sc svd} reconstruction can only approximate
the low frequency information in the initial state whereas the split-frequency
model manages to capture significantly more.
Note that the reconstruction here is better at those places where
$u_0$ has larger magnitude.

If we had to reduce the noise level to $\delta=0.001$ or reduce
the value of $T$, this difference would have been even more apparent.
If we included the single split frequency reconstruction it would show
a significant improvement over the {\sc svd} but clearly poorer than
the split into three bands.
Indeed, 
a similar instance of $u_0$ benefits from a further splitting 
of frequency bands beyond the three level,
see Figure~\ref{fig:reconstructions3}. 

\bigskip
\begin{figure}[ht]
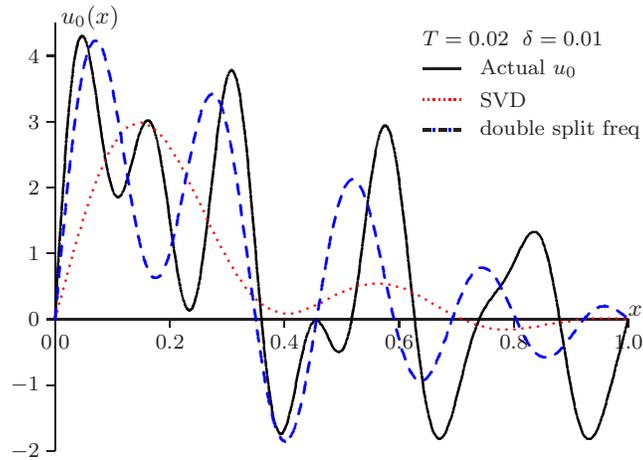

\hbox to \hsize{\hss\box\figurethree\hss}
\caption{\small Reconstructions from SVD and 
double split 
frequency method.}
\label{fig:reconstructions2}
\end{figure}

\bigskip
\begin{figure}[ht]
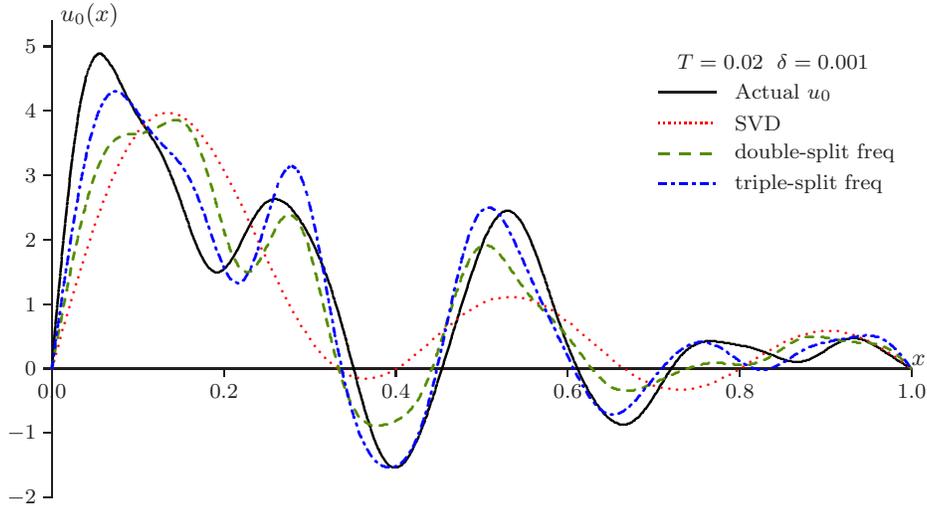

\hbox to \hsize{\hss\box\figurefour\hss}
\caption{\small Reconstructions from 
SVD as well as double and triple split frequency method.}
\label{fig:reconstructions3} 
\end{figure}

\begin{remark}
In higher space dimensions, for special geometries, the eigenfunctions and -values of $-\triangle$ can be computed analytically and everything would proceed as described.
For more complex geometries one can rely on a numerical solver
and there are many possibilities here, that even would include the
multi-term fractional order derivative and more general elliptic operators $\mathbb{L}$, see \cite{JinLazarovLiuZhou:2015}
and references within, or for the combined subdiffusion and
fractional operator in space \eqref{eqn:alpha-beta-pseudoparabolic},
\cite{BonitoLeiPasciak:2017}.
The linearity of the problem with respect to $u_0$ would still allow a
decomposition into frequency bands as described for the case in this section.
Note that only a very limited number of eigenfunctions and -values is required for the reconstruction, since the high frequency part can be tackled directly via the fractional PDE of temporal order $\alpha_2$.
\end{remark}

\section{Convergence analysis}\label{sec:conv}

The goal of this section is to provide a convergence analysis in the sense of regularization, i.e., as the noise level $\delta$ tends to zero, for the split frequency approach from Section~\ref{sec:split_frequency_case}. 
\revision{Here the order of time differentiation $\alpha$ acts as a regularization parameter}
\Margin{Report C (25)}
As a preliminary result we will show convergence of the subdiffusion regularization \eqref{eqn:subdiff-reg}, \eqref{eqn:amplification1}, both with an a priori choice of $\alpha$ and with the discrepancy principle for choosing $\alpha$.

\subsection{Simple subdiffusion regularization}
We first of all consider approximate reconstruction of $u_0$ as
\begin{equation}\label{eqn:u0bsd}
u_0^\delta(\cdot;\alpha) = \sum_{k=1}^\infty 
\revision{\frac{1}{E_{\alpha,1}(-\lambda_k T^\alpha)}}
 \tilde{c}_k^\delta \phi_k
\end{equation}
\Margin{Report C (26)}
that is, the initial data of the solution $u$ to the subdiffusion equation of
order $\alpha$ in equation \eqref{eqn:subdiffusion} with final data $\uTtd$
an  
\revision{$\dH{2}$}
\Margin{Report C (15)}
smoothed version of the given final data and where
 $\tilde{c}_k^\delta$ are its Fourier coefficients.
We will show that $u_0^\delta(\cdot;\alpha)\to u_0$ in $L^2(\Omega)$ as $\delta\to0$ provided $\alpha=\alpha(\delta)$ is appropriately chosen.
\revision{
Note that the relation $c_k = e^{-\lambda_k T} a_k$ holds, which corresponds to the identity $g = \exp(\mathbb{L}T) u_0$.}
\Margin{Report C (27), (36)}

Since most of the proofs here will be set in Fourier space,
we recall the notation
\[
u_0=\sum_{n=1}^\infty a_n \phi_n\,,\quad
\uT=\sum_{n=1}^\infty c_n \phi_n\,,\quad
\uTdel=\sum_{n=1}^\infty c_n^\delta \phi_n\,,\quad
\uTtd=\sum_{n=1}^\infty \tilde{c}_n^\delta \phi_n\,,
\]
where $\{\phi_n\}$ are eigenfunctions of $-\mathbb{L}$ on $\Omega$
with homogeneous Dirichlet boundary conditions and $\{\lambda_n\}$
are the corresponding eigenvalues enumerated according to value.

As a preliminary step, we provide a result on 
\revision{$\dH{2}$}
\Margin{Report C (15)}
(more generally, 
\revision{$\dH{2s}$})
\Margin{Report C (15)}
smoothing of $L^2$ data, which, in view of $H^2-L^2$-wellposedness of time fractional backwards diffusion, is obviously a crucial ingredient of regularization by backwards subdiffusion.

Recall the above mentioned Landweber iteration for defining $\uTtd=w^{(i_*)}$
\begin{equation}\label{eqn:LW}
w^{(i+1)}=w^{(i)}-A(w^{(i)}-\uTdel)\,, \qquad w^{(0)}=0\,,
\end{equation}
where  
\begin{equation}\label{eqn:Asmoothing}
A=\mu (-\mathbb{L})^{-2s}
\end{equation}
with $s\geq1$ and $\mu>0$ chosen so that $\|A\|_{L^2\to L^2}\leq 1$.

\begin{lemma}\label{lem:smoothing}
A choice of 
\begin{equation}\label{eqn:istar}
i_*\sim T^{-2} \log\left(\frac{\|u_0\|_{L^2(\Omega)}}{\delta}\right)
\end{equation}
yields
\begin{equation}\label{eqn:delta_smoothed}
\|\uT-\uTtd\|_{L^2(\Omega)}\leq C_1 \delta\,, \quad
\|\mathbb{L}^s(\uT-\uTtd)\|_{L^2(\Omega)}\leq C_2 \, T^{-1} \, \delta\, \sqrt{\log\left(\frac{\|u_0\|_{L^2(\Omega)}}{\delta}\right)} =:\tilde{\delta}
\end{equation}
for some $C_1,C_2>0$ independent of $T$ and $\delta$. 
\end{lemma}
The proof can be found in the Appendix.
Existing results on convergence of Landweber iteration do not apply here
 due to the infinite order smoothness of the function we are smoothing;
 more precisely, the fact that it satisfies a source condition with an exponentially decaying index function. 

We are now in the position to prove convergence of
$u_0^\delta(\cdot;\alpha)$ in the sense of a regularization method,
first of all with an a priori choice of $\alpha$.

\begin{theorem}\label{prop:reg_apriori}
Let $u_0\in$ 
\revision{$\dH{2(1+1/p)}$})
\Margin{Report C (15)}
for some $p\in(1,\infty)$, and let
$u_0^\delta(\cdot;\alpha)$ be defined by \eqref{eqn:u0bsd} with $\uTtd=w^{(i_*)}$ according to \eqref{eqn:LW}, \eqref{eqn:Asmoothing}, \eqref{eqn:istar}, with 
\revision{$s\geq 1+\tfrac{1}{p}$,} 
and
assume that $\alpha=\alpha(\tilde{\delta})$ is chosen such that
\begin{equation}\label{eqn:apriorialpha}
\alpha(\tilde{\delta})\nearrow 1 \mbox{ and } \frac{\tilde{\delta}}{1-\alpha(\tilde{\delta})}\to0
\,, \quad \mbox{ as } \tilde{\delta}\to0\,, 
\end{equation}

Then 
\[
\|u_0^\delta(\cdot;\alpha(\delta))-u_0\|_{L^2(\Omega)}\to0 
\,, \quad \mbox{ as } \tilde{\delta}\to0\,,
\]
\end{theorem}

\begin{proof}
In terms of Fourier coefficients, the $L^2$ error can be written as 
\begin{equation}\label{eqn:errordecomp}
\begin{aligned}
\|u_0^\delta(\cdot;\alpha(\delta))&-u_0\|_{L^2(\Omega)}
=\Bigl(\sum_{k=1}^\infty 
\Bigl(\frac{1}{E_{\alpha,1}(-\lambda_kT^\alpha)}\tilde{c}_k^\delta - a_k\Bigr)^2
\Bigr)^{1/2}\\
&=
\Bigl(\sum_{k=1}^\infty 
\Bigl(\frac{1}{E_{\alpha,1}(-\lambda_kT^\alpha)}(\tilde{c}_k^\delta-c_k) +
\Bigl(\frac{e^{-\lambda_kT}}{E_{\alpha,1}(-\lambda_kT^\alpha)}-1 \Bigr) a_k\Bigr)^2
\Bigr)^{1/2}\\
&\leq \bar{C}\frac{1}{1-\alpha} 
\bigl(\sum_{k=1}^\infty \lambda_k^2(\tilde{c}_k^\delta-c_k)^2\bigr)^{1/2}
+\left(\sum_{k=1}^\infty w_k(\alpha)^2 a_k^2\right)^{1/2}
\end{aligned}
\end{equation}
where 
\begin{equation}\label{eqn:wk}
w_k(\alpha)=\frac{e^{-\lambda_kT}}{E_{\alpha,1}(-\lambda_kT^\alpha)}-1 
\end{equation} 
and we have used the triangle inequality as well as \eqref{eqn:stability_est}.
The first term on the right hand side is bounded by 
$ \bar{C}\frac{1}{1-\alpha} \tilde{\delta}$,
which tends to zero as $\tilde{\delta}\to0$ under condition
\eqref{eqn:apriorialpha}.
The second term on the right hand side tends to zero as
$\alpha\nearrow1$, since we have, due to \eqref{eqn:boundH2}, 
$$
w_k(\alpha) \to0\mbox{ as }\alpha\nearrow1\mbox{ and }
w_k(\alpha)\leq \tilde{C} 
\revision{\lambda_k^{1+1/p}} 
\mbox{ for all }k\in\mathbb{N}\,.
$$
\Margin{Report C (28)}
From the fact that 
$
\Bigl(\sum_{k=1}^\infty \lambda_k^{2(1+1/p)} a_k^2\Bigr)^{1/2}
\revision{=\|u_0\|_{\dH{2(1+1/p)}}}
<\infty
$
\Margin{Report C (15)}
and Lebesgue's Dominated Convergence Theorem,
we have convergence of the infinite series\\
$\sum_{k=1}^\infty w_k(\alpha)^2 a_k^2$ to zero as $\alpha\nearrow1$.
\end{proof}

We now consider an a posteriori choice of $\alpha$ according to the discrepancy principle, applied to the smoothed data 
\begin{equation}\label{eqn:dpalpha}
\underline{\tau}\tilde{\delta}\leq \|\exp(\mathbb{L}T)u_0^\delta(\cdot;\alpha)-\uTtd\|\leq \overline{\tau}\tilde{\delta}
\end{equation}
for some fixed constants $1<\underline{\tau}<\overline{\tau}$ independent of $\tilde{\delta}$.

The fact that this regularization parameter choice is
well-defined, that is existence of an $\alpha$ such that
\eqref{eqn:dpalpha} holds, can be proven under the assumption
\begin{equation}\label{eqn:snr}
\|\uTtd\|_{L^2}> \hat{\tau} \tilde{\delta}
\end{equation}
with $\hat{\tau}=\overline{\tau}/(1-(1+\lambda_1T)\exp(-\lambda_1T))$ (note that the factor $1-\exp(-\lambda_1T)(1+\lambda_1T)$ is always positive).
Namely, from Lemma~\ref{lem:mlf-asympt-bound} with 
\revision{
$\lim_{x\to1}\Gamma(x)=1$ we
conclude that $w_k(\alpha)$ as defined in \eqref{eqn:wk} satisfies 
$\lim_{\alpha\to0} w_k(\alpha)=e^{-\lambda_kT}(1+\lambda_kT)-1$
and $w_k(\alpha)\leq \tilde{C} \lambda_k^{1+1/p}$ with $\sum_{k=1}^\infty \Bigl(\lambda_k^{1+1/p}\tilde{c}_k^\delta\Bigr)^2=\|\uTtd\|_{\dH{2(1+1/p)}}<\infty$, thus by Lebesgue's Dominated Convergence Theorem,
\begin{equation*}
\begin{aligned}
\lim_{\alpha\to0} \|\exp(\mathbb{L}T)u_0^\delta(\cdot;\alpha)-\uTtd\|_{L^2}
&= \lim_{\alpha\to0} \Bigl(\sum_{k=1}^\infty \Bigl(w_k(\alpha)\, \tilde{c}_k^\delta\Bigr)^2
\Bigr)^{1/2}\\
&= \Bigl(\sum_{k=1}^\infty \Bigl(e^{-\lambda_kT}(1+\lambda_kT)-1\Bigr)^2(\tilde{c}_k^\delta)^2
\Bigr)^{1/2}\\
&\geq (1-e^{-\lambda_1T}(1+\lambda_1T)) 
\Bigl(\sum_{k=1}^\infty \tilde{c}_k^\delta\Bigr)^2\\
&= (1-e^{-\lambda_1T}(1+\lambda_1T)) \|\uTtd\|_{L^2} 
> \overline{\tau}\tilde{\delta}\,.
\end{aligned}
\end{equation*}
}
\Margin{Report C (29)}
On the other hand, 
\begin{equation*}
\lim_{\alpha\to1} \|\exp(\mathbb{L}T)u_0^\delta(\cdot;\alpha)-\uTtd\|_{L^2}
=0
\leq\underline{\tau}\tilde{\delta}\,.
\end{equation*}
Hence, from continuity of the mapping $\alpha\mapsto \|\exp(\mathbb{L}T)u_0^\delta(\cdot;\alpha)-\uTtd\|_{L^2}$ on the interval $(0,1)$ (which would actually not hold on $[0,1)$!) and the Intermediate Value Theorem, we conclude existence of $\alpha\in(0,1)$ such that \eqref{eqn:dpalpha} holds.

Note that the case of condition \eqref{eqn:snr} being violated for all $\delta>0$ sufficiently small is trivial in the sense that then obviously $u_0=0$ holds.

\begin{theorem}\label{prop:reg_aposteriori}
Let $u_0\in 
\revision{\dH{2(1+2/p)}}
$ 
\Margin{Report C (15)}
for some $p\in(1,\infty)$, and let
$u_0^\delta(\cdot;\alpha)$ be defined by \eqref{eqn:u0bsd} with $\uTtd=w^{(i_*)}$
according to \eqref{eqn:LW}, \eqref{eqn:Asmoothing}, \eqref{eqn:istar}, with $s\geq 1+\frac{1}{p}$, and
assume that $\alpha=\alpha(\uTtd,\tilde{\delta})$ is chosen according to \eqref{eqn:dpalpha}.
Then 
\begin{equation*}
u_0^\delta(\cdot;\alpha(\uTtd,\tilde{\delta}))
\rightharpoonup u_0 \mbox{ in }L^2(\Omega) 
\,, \quad \mbox{ as }\delta\to0\,.
\end{equation*}
\end{theorem}

\begin{proof}
In view of the representation
\begin{equation*}
\begin{aligned}
&\|\exp(\mathbb{L}T)u_0^\delta(\cdot;\alpha)-\uTdel\|
=\Bigl(\sum_{k=1}^\infty 
\Bigl(w_k(\alpha) \tilde{c}_k^\delta\Bigr)^2
\Bigr)^{1/2}\\
&\qquad\qquad\qquad\leq
\Bigl(\sum_{k=1}^\infty \Bigl(w_k(\alpha) e^{-\lambda_kT} a_k\Bigr)^2
\Bigr)^{1/2}
+\Bigl(\sum_{k=1}^\infty 
\Bigl(w_k(\alpha) (\tilde{c}_k^\delta-c_k)\Bigr)^2
\Bigr)^{1/2}\,,
\end{aligned}
\end{equation*}
and likewise 
\begin{equation*}
\|\exp(\mathbb{L}T)u_0^\delta(\cdot;\alpha)-\uTdel\|
\geq
\Bigl(\sum_{k=1}^\infty \Bigl(w_k(\alpha) e^{-\lambda_kT} a_k\Bigr)^2
\Bigr)^{1/2}
-\Bigl(\sum_{k=1}^\infty 
\Bigl(w_k(\alpha) (\tilde{c}_k^\delta-c_k)\Bigr)^2
\Bigr)^{1/2}\,,
\end{equation*}
as well as \eqref{eqn:boundH2}, which 
yields 
\Margin{Report C (30)}
\begin{equation}\label{eqn:wkalphack}
\Bigl(\sum_{k=1}^\infty 
\Bigl(w_k(\alpha) (\tilde{c}_k^\delta-c_k)\Bigr)^2
\Bigr)^{1/2}
\leq \tilde{C} \|\uTtd-\uT\|_{
\revision{\dH{2(1+1/p)}}
},
\end{equation}
\Margin{Report C (15)}
the discrepancy principle \eqref{eqn:dpalpha} yields
\begin{equation}\label{eqn:deltaest}
(\underline{\tau}-\tilde{C})\tilde{\delta}\leq
\Bigl(\sum_{k=1}^\infty 
\Bigl(w_k (\alpha) e^{-\lambda_kT} a_k\Bigr)^2
\Bigr)^{1/2}
\leq (\overline{\tau}+\tilde{C})\tilde{\delta}\,.
\end{equation}
From the error decomposition \eqref{eqn:errordecomp} we therefore conclude
\Margin{Report C (31)}
\begin{equation*}
\begin{aligned}
&\|u_0^\delta(\cdot;\alpha
)-u_0\|_{L^2(\Omega)}
\leq \bar{C}\frac{\tilde{\delta}}{1-\alpha} 
+\left(\sum_{k=1}^\infty \Bigl(w_k(\alpha) a_k\Bigr)^2\right)^{1/2}\\
&\leq \frac{\bar{C}}{\underline{\tau}-\tilde{C}} 
\Bigl(\sum_{k=1}^\infty 
\Bigl(\frac{w_k (\alpha) e^{-\lambda_kT}}{1-\alpha} a_k\Bigr)^2
\Bigr)^{1/2}
+\left(\sum_{k=1}^\infty \Bigl(w_k(\alpha) a_k\Bigl)^2\right)^{1/2}\,,
\end{aligned}
\end{equation*}
where due to \eqref{eqn:rateE} and \eqref{eqn:boundH2} we have
\begin{equation*}
\begin{aligned}
&\left|\frac{w_k (\alpha) e^{-\lambda_kT}}{1-\alpha}\right| 
=\left|\frac{e^{-\lambda_kT}}{E_{\alpha,1}(-\lambda_kT^\alpha)}-1 \right| \frac{e^{-\lambda_kT}}{1-\alpha}\\
&\qquad\qquad=\left|e^{-\lambda_kT}-E_{\alpha,1}(-\lambda_kT^\alpha)\right|
\frac{e^{-\lambda_kT}}{E_{\alpha,1}(-\lambda_kT^\alpha)} \frac{1}{1-\alpha}
\leq C \lambda_k^{1/p} \bigl(1+\tilde{C} \lambda_k^{1+1/p}\bigr) \,.
\end{aligned}
\end{equation*}
Taking into account the assumption $u_0\in 
\revision{\dH{2(1+2/p)}}
$,
\Margin{Report C (15)}
we get that $u_0^\delta(\cdot;\alpha)$ is uniformly bounded in $L^2(\Omega)$ and thus has a weakly convergent subsequence whose limit due to the upper estimate in \eqref{eqn:deltaest} has to coincide with $u_0$. A subsequence-subsequence argument therefore yields weak $L^2$ convergence of $u_0^\delta(\cdot;\alpha)^\delta$ to $u_0$.
\end{proof}

Concerning convergence rates, observe, first of all, that the rate and stability estimates \eqref{eqn:rateE}, \eqref{eqn:boundH2} yields the following convergence rate for the time fractional reconstruction in case of very smooth data $u_0$ and noise free data.
\[
\begin{aligned}
\|u_0^0(\cdot;\alpha)-u_0\|_{L^2(\Omega)} &\leq \sqrt{\sum_{k=1}^\infty  \left(\left(\frac{\exp(-\lambda_k T)}{E_{\alpha,1}(-\lambda_k T^\alpha)}-1\right)a_k\right)^2}\\
&\leq C(1-\alpha)
\sqrt{\sum_{k=1}^\infty \left(\exp(\lambda_k T) (1+\tilde{C}\lambda_k^{1+1/p}) \lambda_k^{1/p} a_k\right)^2}
\end{aligned}
\]
where $a_k$ are the Fourier coefficients of the initial data and hence the right hand side is a very strong norm of $u_0$.
\\
Convergence rates under weaker norm bounds on $u_0$ and with noisy data can be obtained similarly to \cite{Hohage97} by means of Jensen's inequality and an appropriate choice of $\alpha$. 
\begin{theorem}\label{prop:reg_rate}
Let $u_0\in 
\revision{\dH{2(1+1/p+\max\{1/p,q\})}}
$ 
\Margin{Report C (15)}
for some $p\in(1,\infty)$, $q>0$, and let
$u_0^\delta(\cdot;\alpha)$ be defined by \eqref{eqn:u0bsd} with $\uTtd=w^{(i_*)}$ according to \eqref{eqn:LW}, \eqref{eqn:Asmoothing}, \eqref{eqn:istar}, with $s\geq 1$, and
assume that $\alpha=\alpha(\tilde{\delta})$ is chosen such that
\begin{equation}\label{eqn:apriorialpha_rate}
1-\alpha(\tilde{\delta})\sim \sqrt{\tilde{\delta}}
\,, \quad \mbox{ as } \tilde{\delta}\to0\,, 
\end{equation}
Then 
\begin{equation}\label{eqn:rate}
\|u_0^\delta(\cdot;\alpha(\tilde{\delta}))-u_0\|_{L^2(\Omega)} = O\left(\log(\tfrac{1}{\delta})^{-2q}\right)
\,, \quad \mbox{ as } \delta\to0\,.
\end{equation}
In the noise free case we have 
\begin{equation}\label{eqn:ratedel0}
\|u_0^0(\cdot;\alpha)-u_0\|_{L^2(\Omega)} = O\left(\log(\tfrac{1}{1-\alpha})^{-2q}\right) 
\,, \quad \mbox{ as }\alpha\nearrow 1\,.
\end{equation}
\end{theorem}
\begin{proof}
For $c=1-\lambda_1T$ and some $q>0$ set
\[
\begin{aligned}
&f(x):=\Bigl(c-\log(x)\Bigr)^{-q} \mbox{ for }x\in(0,\exp(-\lambda_1 T)] \\
&\varphi(\xi):=\xi \exp\Bigl(2\Bigl(c-\xi^{-\frac{1}{2q}}\Bigr)\Bigr) \mbox{ for } \xi\in (0,1]\,,
\end{aligned}\]
so that 
\[
f(x)\in(0,1]\mbox{ and }\quad \varphi(f^2(x))=x^2f^2(x)\mbox{ for }x\in(0,\exp(-\lambda_1 T)]\,.
\]
It is readily checked that $\varphi$ is convex and strictly monotonically increasing, and that the values of its inverse can be estimated as follows
\begin{equation}\label{estphiinv}
\begin{aligned}
b\varphi^{-1}(\tfrac{a}{b})\leq \left(\frac{2c+\log(\tfrac{1}{a})}{2B^{\frac{1}{2q}}+C_q}\right)^{-2q}
\mbox{ for }\tfrac{a}{b}\in(0,e^{2\lambda_1T}]\,,\quad b\in(0,B]\,,\\
\end{aligned}
\end{equation}
where $C_q>0$ is chosen such that 
\[
\log(z)\leq C_qz^{\frac{1}{2q}}\mbox{ for }z\geq
\revision{\tfrac{1}{B}}\,.
\]
\revision{
Estimate \eqref{estphiinv} can be verified by the following chain of implications and estimates
\[
\begin{aligned}
&\xi=\varphi^{-1}(\tfrac{a}{b})
\ \Leftrightarrow \
\tfrac{a}{b}=\varphi(\xi)=\xi \exp\Bigl(2\Bigl(c-\xi^{-\frac{1}{2q}}\Bigr)\Bigr)
\\& \Leftrightarrow \
\log(a)=\log(b\xi)+2c-2b^{\frac{1}{2q}} (b\xi)^{-\frac{1}{2q}}
\\& \Leftrightarrow \
2c +\log(\tfrac{1}{a})=\log(\tfrac{1}{b\xi})+2b^{\frac{1}{2q}} (b\xi)^{-\frac{1}{2q}}
\leq (C_q+2B^{\frac{1}{2q}})(b\xi)^{-\frac{1}{2q}}
\\& \Leftrightarrow \
b\xi\leq \left(\frac{2c+\log(\tfrac{1}{a})}{2B^{\frac{1}{2q}}+C_q}\right)^{-2q}
\end{aligned}
\]
}
\Margin{Report C (32)}

Therefore, Jensen's inequality yields, for any two sequences 
$(\sigma_k)_{k\in\mathbb{N}}$\\ $\subseteq(0,\exp(-\lambda_1 T)]$ and $(\omega_k)_{k\in\mathbb{N}}\in\ell^2$,
\[
\varphi\left(\frac{\sum_{k=1}^\infty \Bigl(f(\sigma_k)\omega_k\Bigr)^2}{\sum_{k=1}^\infty \omega_k^2}\right)
\leq \frac{\sum_{k=1}^\infty \varphi(f(\sigma_k)^2)\omega_k^2}{\sum_{k=1}^\infty \omega_k^2}
=\frac{\sum_{k=1}^\infty \Bigl(f(\sigma_k)\sigma_k\omega_k\Bigr)^2}{\sum_{k=1}^\infty \omega_k^2}\,.
\]
Hence, applying $\varphi^{-1}$ to both sides and using \eqref{estphiinv}, we obtain
\begin{equation}\label{eqn:Jensen}
\sum_{k=1}^\infty \Bigl(f(\sigma_k)\omega_k\Bigr)^2\leq 
\sum_{k=1}^\infty \omega_k^2 \quad \varphi^{-1}\left(\frac{\sum_{k=1}^\infty \Bigl(f(\sigma_k)\sigma_k\omega_k\Bigr)^2}{\sum_{k=1}^\infty \omega_k^2}\right)
\leq \left(
\revision{
\frac{2c+\log(\tfrac{1}{a})}{2B^{\frac{1}{2q}}+C_q}
}
\right)^{-2q}
\end{equation}
\Margin{Report C (32)}
for
\[
a=\sum_{k=1}^\infty \Bigl(f(\sigma_k)\sigma_k\omega_k\Bigr)^2\,, \quad
b=\sum_{k=1}^\infty \omega_k^2\,.
\] 
Setting 
\[
\begin{aligned}
&\sigma_k:=\exp(-\lambda_k T)\,,\\
&\omega_k:=\left(\frac{\exp(-\lambda_k T)}{E_{\alpha,1}(-\lambda_k T^\alpha)}-1\right)\frac{1}{f(\sigma_k)}\, a_k
=\left(\frac{\exp(-\lambda_k T)}{E_{\alpha,1}(-\lambda_k T^\alpha)}-1\right)
(c+\lambda_kT)^q\, a_k\, 
\end{aligned}
\] 
so that $\sum_{k=1}^\infty \Bigl(f(\sigma_k)\omega_k\Bigr)^2=\|u_0^0(\cdot;\alpha)-u_0\|_{L^2(\Omega)}^2$ and by \eqref{eqn:rateE}, \eqref{eqn:boundH2}
\[
\begin{aligned}
&a=\sum_{k=1}^\infty \left(\left(\frac{\exp(-\lambda_k T)}{E_{\alpha,1}(-\lambda_k T^\alpha)}-1\right) \exp(-\lambda_k T) \, a_k\right)^2 \\
&\qquad\leq C^2(1-\alpha)^2
\sum_{k=1}^\infty \left((1+\tilde{C}\lambda_k^{1+1/p}) \lambda_k^{1/p}\, a_k\right)^2
\\
&b=\sum_{k=1}^\infty \omega_k^2\leq \tilde{C}^2
\sum_{k=1}^\infty \left(\lambda_k^{1+1/p} (c+\lambda_kT)^q\, a_k\right)^2
=:B\,,
\end{aligned}
\]
we deduce from \eqref{eqn:Jensen} the rate \eqref{eqn:ratedel0}.
\\
The rate \eqref{eqn:rate} with noisy data follows from the error decomposition \eqref{eqn:errordecomp} 
\revision{using the fact that the second term in \eqref{eqn:errordecomp} just coincides with  
$\|u_0^0(\cdot;\alpha)-u_0\|_{L^2(\Omega)}$, for which we can make use of \eqref{eqn:ratedel0},
}
\Margin{Report C (33)}
\[
\|u_0^\delta(\cdot;\alpha)-u_0\|_{L^2(\Omega)}\leq\bar{C}\frac{\tilde{\delta}}{1-\alpha} + O\left(\log(\tfrac{1}{1-\alpha})^{-2q}\right)\,,
\]
together with the parameter choice \eqref{eqn:apriorialpha_rate}.
\end{proof}

\subsection{Split frequency subdiffusion regularization} \label{sec:conv_split}
Our actual goal is to establish convergence in the sense of a regularization method of the split frequency subdiffusion reconstruction
\begin{equation}\label{eqn:u0bsd_split}
\begin{aligned}
u_0^\delta(\cdot;\alpha,K) 
&= u_{0,lf}^\delta(\cdot;K)+u_{0;hf}^\delta(\cdot;\alpha,K)\\
&=\sum_{k=1}^K \exp(-\lambda_k T) c_k^\delta \phi_k
+\sum_{k=K+1}^\infty E_{\alpha,1}(-\lambda_k T^\alpha) \tilde{c}_k^\delta \phi_k
\end{aligned}
\end{equation}
Initially  $K$ is determined by the discrepancy principle 
\begin{equation}\label{eqn:dpK}
K=\min\{ k\in\mathbb{N}\, : \, \|\exp(\mathbb{L}T)u_{0,lf}^\delta-\uTdel\|\leq \tau\delta\}
\end{equation}
for some fixed $\tau>1$.
This determines the low frequency part $u_{0,lf}^\delta(\cdot;K)$. After this is done, $u_{0,hf}^\delta(\cdot;\alpha,K)$ is computed, with $\alpha$ calibrated according to the discrepancy principle 
\begin{equation}\label{eqn:dpalpha_split}
\underline{\tau}\tilde{\delta}\leq \|\exp(\mathbb{L}T)u_0^\delta(\cdot;\alpha,K)-\uTdel\|\leq \overline{\tau}\tilde{\delta}
\end{equation}

\begin{theorem}\label{prop:reg_aposteriori_split}
Let $u_0\in 
\revision{\dH{2(1+2/p)}}
$ 
\Margin{Report C (15)}
for some $p\in(1,\infty)$, and let
$u_0^\delta(\cdot;\alpha,K)$ be defined by \eqref{eqn:u0bsd_split} with $\uTtd=w^{(i_*)}$ according to \eqref{eqn:LW}, \eqref{eqn:Asmoothing}, \eqref{eqn:istar}, with $s\geq 1+\frac{1}{p}$, and
assume that $K=K(\uTdel,\delta)$ and $\alpha=\alpha(\uTtd,\tilde{\delta})$  are chosen according to \eqref{eqn:dpK} and \eqref{eqn:dpalpha_split}.
Then 
\begin{equation*}
u_0^\delta(\cdot;\alpha(\uTtd,\tilde{\delta}),K(\uTdel,\delta))\rightharpoonup u_0 \mbox{ in }L^2(\Omega) 
\,, \quad \mbox{ as } \delta\to0\,.
\end{equation*}
\end{theorem}

\begin{proof}
The discrepancy principle \eqref{eqn:dpK} for $K$ in terms of Fourier coefficients reads as 
\[
\Bigl(\sum_{k=K+1}^\infty(c_k^\delta)^2\Bigr)^{1/2}\leq \tau\delta\leq
\Bigl(\sum_{k=K}^\infty(c_k^\delta)^2\Bigr)^{1/2}
\]
which due to the fact that $c_k^\delta=c_k^\delta-c_k+e^{-\lambda_kT}a_k$ and the triangle inequality, as well as \eqref{eqn:delta} implies
\begin{equation}\label{eqn:dpKest}
\Bigl(\sum_{k=K}^\infty\Bigl(e^{-\lambda_kT}a_k\Bigr)^2\Bigr)^{1/2}\geq (\tau-1)\delta
\ \mbox{ and }\
\Bigl(\sum_{k=K+1}^\infty\Bigl(e^{-\lambda_kT}a_k\Bigr)^2\Bigr)^{1/2}\leq (\tau+1)\delta
\end{equation}
From the discrepancy principle \eqref{eqn:dpalpha_split} for $\alpha$ we
conclude
\[
\underline{\tau}\tilde{\delta}\leq 
\Bigl(\sum_{k=1}^K(c_k^\delta-\tilde{c}_k^\delta)^2+
\sum_{k=K+1}^\infty\Bigl(w_k(\alpha)\tilde{c}_k^\delta\Bigr)^2\Bigr)^{1/2}\leq
\overline{\tau}\tilde{\delta}\,,
\]
where again we can use $\tilde{c}_k^\delta=\tilde{c}_k^\delta-c_k+e^{-\lambda_kT}a_k$ and the triangle inequality, as well as \eqref{eqn:delta}, \eqref{eqn:delta_smoothed} and \eqref{eqn:boundH2} 
\Margin{Report C (34)}
(cf. \eqref{eqn:wkalphack}) to conclude 
\[
(\underline{\tau}-\tilde{C})\tilde{\delta}-(1+C_1)\delta\leq
\sum_{k=K+1}^\infty\Bigl(w_k(\alpha)e^{-\lambda_kT}a_k\Bigr)^2\Bigr)^{1/2}\leq
(\underline{\tau}+\tilde{C})\tilde{\delta}+(1+C_1)\delta\,,
\]
where  $ (1+C_1)\delta\leq\tilde{C}_1\tilde{\delta}\,$.

For the error in the initial data this yields
\begin{equation*}
\begin{aligned}
&\|u_0^\delta(\cdot;\alpha,K)-u_0\|_{L^2(\Omega)}\\
&=\Bigl(
\sum_{k=1}^K\Bigl(e^{\lambda_kT}c_k^\delta-a_k\Bigr)^2
+\sum_{k=K+1}^\infty 
\Bigl(\frac{1}{E_{\alpha,1}(-\lambda_kT^\alpha)}\tilde{c}_k^\delta - a_k\Bigr)^2
\Bigr)^{1/2}\\
&=\Bigl(
\sum_{k=1}^K\Bigl(e^{\lambda_kT}(c_k^\delta-c_k)\Bigr)^2
+\sum_{k=K+1}^\infty 
\Bigl(
\frac{1}{\lambda_k E_{\alpha,1}(-\lambda_kT^\alpha)}\lambda_k(\tilde{c}_k^\delta-c_k)
+w_k(\alpha)a_k\Bigr)^2
\Bigr)^{1/2}\\
&\leq
e^{\lambda_KT}\delta
+\frac{\bar{C}}{1-\alpha}\tilde{\delta}
+\left(\sum_{k=K+1}^\infty \Bigl(w_k(\alpha) a_k\Bigr)^2\right)^{1/2}\\
&\leq 
\frac{1}{\tau-1} \Bigl(\sum_{k=K}^\infty\Bigl(e^{(\lambda_K-\lambda_k)T}a_k\Bigr)^2\Bigr)^{1/2}
+
\frac{\bar{C}}{\underline{\tau}-\tilde{C}-\tilde{C_1}} 
\Bigl(\sum_{k=K+1}^\infty 
\Bigl(\frac{w_k (\alpha) e^{-\lambda_kT}}{1-\alpha} a_k\Bigr)^2
\Bigr)^{1/2}\\
&\qquad+\left(\sum_{k=K+1}^\infty \Bigl(w_k(\alpha) a_k\Bigl)^2\right)^{1/2}\,.
\end{aligned}
\end{equation*}
Since the right hand side estimate in \eqref{eqn:dpKest} implies the convergence
\[
\sum_{k=K}^\infty\Bigl(e^{(\lambda_K-\lambda_k)T}a_k\Bigr)^2\leq
\sum_{k=K}^\infty(a_k)^2 \to 0 \mbox{ as }\delta\to0
\]
\Margin{Report C (35)}
of the first term, the rest of the proof goes analogously to the one of Theorem \ref{prop:reg_aposteriori}.
\end{proof}

\section*{Appendix}\label{sec:appendix}

\begin{proof}[\unskip\nopunct]
{\it Proof of Lemma}~\ref{lem:smoothing}.
The iteration error can be written as
\[
\begin{aligned}
w^{(i)}-\uT&=-(I-A)^i \uT + \sum_{j=0}^{i-1}(I-A)^jA(\uTdel-\uT)\\
&=-(I-A)^i\exp(-T
\revision{
\mu^{1/2s}A^{-1/2s}
}
) u_0 + (I-(I-A)^i)(\uTdel-\uT)\,.
\end{aligned}
\]
\Margin{Report C (36)}
To estimate the 
\revision{$\dH{2s}$}
\Margin{Report C (15)}
 error $\|\mathbb{L}^s(\uT-\uTtd)\|_{L^2(\Omega)}$, we need to consider
\[\begin{aligned}
&\|A^{-1/2}(w^{(i)}-\uT)\|_{L^2(\Omega)}\\
&\leq
\|(I-A)^i A^{-1/2}\exp(-T
\revision{
\mu^{1/2s}
}
A^{-1/(2s)})\|_{L^2(\Omega)\to L^2(\Omega)} \|u_0\|_{L^2(\Omega)} \\
&\qquad+ \|A^{-1/2}(I-(I-A)^i)\|_{L^2(\Omega)\to L^2(\Omega)}\,\delta\,,
\end{aligned}\]
where the terms on the right hand side (approximation error and propagated noise) can be estimated using spectral theory:
\Margin{Report C (37)}
\[
\begin{aligned}
&\|(I-A)^i A^{-1/2}\exp(-T
\revision{
\mu^{1/2s}
}
A^{-1/(2s)})\|_{L^2(\Omega)\to L^2(\Omega)}\leq
\sup_{\sigma\in(0,1]} \psi(\sigma)\,, \\
&\|A^{-1/2}(I-(I-A)^i
\revision{)}
\|_{L^2(\Omega)\to L^2(\Omega)}\leq
\sup_{\sigma\in(0,1]} \chi(\sigma)
\end{aligned}
\]
(note that $\sigma$ corresponds to $(\sqrt{\mu}\lambda_k^{-1})^s$) for 
\[
\psi(\sigma)=(1-\sigma^2)^i\sigma^{-1} \exp(-\sigma^{-1/s}T) \,, \qquad \chi(\sigma)=\sigma^{-1}(1-(1-\sigma^2)^i)\,,
\]
where the functions $\psi$ and $\chi$ can be bounded as follows.

\paragraph{Bounding $\psi$:} 
Since for $\sigma\in(0,1]$ and $s\geq1$, we have $\sigma^{1/s}\geq \sigma$ and therefore 
$(1-\sigma^2)^i\sigma^{-1}
\revision{\geq} 
(1-\sigma^{2/s})^i\sigma^{-1/s}$, 
\Margin{Report C (38)}
with the transformation of variables $\sigma^{1/s}\to\sigma$ it
suffices to consider the special case $s=1$. 
\revision{
Moreover, 
}
\Margin{Report C (39)}
since $\lim_{\sigma\to0} \psi(\sigma)=0$ and $\psi(1)=0$, it
is enough to consider critical points:
\[
\begin{aligned}
&\psi'(\sigma)
= -(1-\sigma^2)^{i-1}\sigma^{-3}\exp(-\sigma^{-1}T)\bigl((2i-1)\sigma^3+T\sigma^2+\sigma-T\bigr)=0
\end{aligned}
\]
Cardano's formula with $a=2i-1$, $b=T$, $c=1$, $d=-T$ (keeping in mind that $a\geq1$, $T\leq1$)
\[
\begin{aligned}
&Q = \dfrac {3 a c - b^2} {9 a^2}= \dfrac {3 a - T^2} {9 a^2}\geq \frac{2}{9a}>0\,, \\
&R = \dfrac {9 a b c - 27 a^2 d - 2 b^3} {54 a^3} = \dfrac {9 a T + 27 a^2 T - 2 T^3} {54 a^3}\geq\frac{T}{2a}>0\,,\\
&Q^3 + R^2=\frac{d^2}{4a^2}+\frac{2c^3-9bcd}{54a^3}-\frac{b^2c^2-9bcd}{108a^4}
\leq \frac{d^2}{4a^2}+\frac{2c^3-9bcd}{54a^3}\\
&\qquad\qquad=\frac{4+18T^2+27T^2a}{108a^3}\leq\frac{22+27T^2a}{108a^3}\\
&\frac{2^{1/3}R}{\sqrt[3]{Q^3+R^2}}\geq \frac{3T}{\sqrt[3]{22+27T^2a}}
\end{aligned}
\]
yields the unique (since $D = Q^3 + R^2 >0$) real root
\[
\begin{aligned}
\sigma_*
&= \sqrt [3] {\sqrt{Q^3 + R^2}+R} - \sqrt [3] {\sqrt{Q^3 + R^2}-R} - \dfrac b {3 a} \\
&= \frac{2R}{\left(\sqrt{Q^3 + R^2}+R\right)^{2/3}+\left(\sqrt{Q^3 + R^2}-R\right)^{2/3} +Q}
 - \dfrac b {3 a} \\
&\geq \frac{2R}{\left(2\sqrt{Q^3 + R^2}+R\right)^{2/3} +Q}
 - \dfrac b {3 a} 
\geq \frac{2R}{(1+2^{2/3})(Q+R^{2/3})}
 - \dfrac b {3 a} \\
& 
\geq \frac{2^{1/3}}{1+2^{2/3}}\frac{R}{\sqrt[3]{Q^3+R^2}}
 - \dfrac b {3 a} \\
&\geq 
\frac{T}{\sqrt[3]{22+27T^2a}}
 - \dfrac T {3 a} \\
&\geq T\cdot
\begin{cases}
\frac{1}{\sqrt[3]{23}}-\frac13=:c_1
&\mbox{ if } 27T^2a\leq 1\\
\frac{1}{\sqrt[3]{23\cdot 27T^2a}}-\frac1{3a}\geq \frac{c_2}{\sqrt[3]{T^2a}}&
\mbox{ if } 27T^2a> 1
\end{cases}\\
&\geq c^{-1}  T \min\{1, 1/\sqrt[3]{T^2a}\}
=:\underline{\sigma}_*\,, 
\end{aligned}
\]
where the identity on the second line follows by multiplication with the denominator of the 2nd line using the identities $(x-y)(x^2+xy+y^2)=x^3-y^3$ and $(x-y)(x+y)=x^2-y^2$.
The estimates on the third and fourth line use the fact that for $x,y>0$ and $q\in(0,1], r\in[1,\infty)$, we have $(x+y)^q\leq x^q+y^q$ and  
\revision{
$x+y\leq 2^{1-1/r}(x^r+y^r)^{1/r}$,
more precisely, with $q=\frac23$ in the 3rd line and with $r=3$ in the 4th line.
}
\Margin{Report C (40)}

Since 
\[
\psi(\sigma)\leq\tilde{\psi}(\sigma):=\sigma^{-1} \exp(-\sigma^{-1}T)
\]
and $\tilde{\psi}'(\sigma)=\exp(-\sigma^{-1}T)(-\sigma^{-2}+\sigma^{-2}T)\leq 0$ for $T\leq1$ we have 
\[
\sup_{\sigma\in(0,1]}\psi(\sigma)=\psi(\sigma_*)\leq\tilde{\psi}(\sigma_*)\leq\tilde{\psi}(\underline{\sigma}_*)
\]

\paragraph{Bounding $\chi$:}
\[
\begin{aligned}
\chi(\sigma)
&=\sqrt{(1-(1-\sigma^2)^i)}\sqrt{\sigma^{-2}(1-(1-\sigma^2)^i)}\\
&=\sqrt{(1-(1-\sigma^2)^i)}\Bigl({\sum_{j=0}^{i-1}(1-\sigma^2)^j}\Bigr)^{1/2}
\leq\sqrt{i}
\end{aligned}
\]
\noindent
These bounds on $\psi$ and $\chi$ yield
\begin{equation}\label{eqn:estLW}
\begin{aligned}
\mu^{-1/2}&\|\mathbb{L}^s(w^{(i)}-\uT)\|_{L^2(\Omega)}\\
&\leq
(c /T) \, \max\{1, z\} \exp\bigl(-c \max\{1, z\}\bigr)\|u_0\|_{L^2(\Omega)}+\sqrt{i}\delta\\
&\mbox{ for }z=\sqrt[3]{iT^2}
\end{aligned}
\end{equation}
for some $c$ independent of $i$ and $T$.
For driving the first term on the right hand side to zero as $\delta\to0$, we need to choose $i_*=i_*(\delta)\to\infty$, thus we will have $\max\{1, \sqrt[3]{i_*(\delta)T^2}\}=\sqrt[3]{i_*(\delta)T^2}$ for $\delta$ sufficiently small. Taking this into account while balancing the two terms on the right hand side of \eqref{eqn:estLW} yields
\[
\frac{\delta}{\|u_0\|_{L^2(\Omega)}}= c z^{-1/2} \exp(-cz)\mbox{ i.e., }
\log\left(\frac{\|u_0\|_{L^2(\Omega)}}{\delta}\right)=cz+\frac12 \log z-\log c\,.
\]
Thus an optimal choice is given by 
\[
i_*\sim T^{-2} \log\left(\frac{\|u_0\|_{L^2(\Omega)}}{\delta}\right)
\]
and yields \eqref{eqn:delta_smoothed}.
\end{proof}

\section*{Acknowledgment}

\noindent
The work of the first author was supported by the Austrian Science Fund FWF
under the grants I2271 and P30054 as well as partially by the
Karl Popper Kolleg ``Modeling-Simulation-Optimization'',
funded by the Alpen-Adria-Universit\" at Klagenfurt and by the
Carin\-thian Economic Promotion Fund (KWF).

\noindent
The work of the second author was supported 
in part by the
National Science Foundation through award DMS-1620138.

\medskip

\revision{
The authors wish to thank the reviewers for their careful reading of the manuscript and their detailed reports with valuable comments and suggestions that have led to an improved version of the paper.
}

\bigskip


\end{document}